\documentclass[reqno, 12pt]{amsart}
\usepackage{amsfonts,epsfig}
\usepackage{latexsym}
\usepackage{amssymb}
\usepackage{amsmath}
\usepackage{amsthm}
\usepackage{graphics}
\usepackage{hyperref}
\def\C{{\mathbb C}}

\def\Q{{\mathbb Q}}
\def\Qp{{\mathbb Q}_p}
\def\Z{{\mathbb Z}}
\def\Zp{{\mathbb Z}_p}
\def\R{{\mathbb R}}

\def\Gal{\operatorname{Gal}}

\def\Hom{\operatorname{Hom}}
\def\Maps{\operatorname{Maps}}
\def\Span{\operatorname{Span}}
\def\sgn{\operatorname{sgn}}

\def\F{{\mathbb F}}
\def\Qp {{{\mathbb Q}_p}}
\def\Fq{{{\mathbb F}_q}}

\newtheorem{lemma}{Lemma}
\newtheorem{cor}{Corollary}
\newtheorem{prop}{Proposition}

\newtheorem{theorem}{Theorem}

\theoremstyle{definition}
\newtheorem{defn}{Definition}

\theoremstyle{remark}
\newtheorem{rem}{Remark}        
\newtheorem{rems}{Remarks}      
\newtheorem{example}{Example}
\newtheorem{examples}{Examples}

\begin{document}
\title[Digit permutations revisted]
{Digit permutations revisited}
\author{David Goss}
\address{Department of Mathematics\\The Ohio State University\\231
W.\ $18^{\rm th}$ Ave.\\Columbus, Ohio 43210}

\email{dmgoss@gmail.com}

\date{Spring, 2016}

\begin{abstract}
We discuss here characteristic $p$ $L$-series as well as the group $S_{(q)}$ which appears to act as symmetries of these
functions. We explain various actions of
$S_{(q)}$ that arise naturally in the theory as well as extensions of these actions. In general such extensions appear to be highly
arbitrary but in the case where the zeroes are unramified, the extension is unique (and it is reasonable to expect it is unique only in this
case). Having unramified zeroes is the best one could hope
for in finite characteristic and appears to be an avatar of the Riemann hypothesis in this setting; see Section  \ref{rh} for a more detailed
discussion.
\end{abstract}

\maketitle

\section{Introduction}\label{intro} The publication of V.G.\ Drinfeld's seminal paper {\it Elliptic modules} \cite{Dr74} in 1974 has lead to the 
emergence of a true ``analytic theory of numbers" involving analysis in finite characteristic. This work, together with previously little known
work by L.\ Carlitz starting in the 1930's (see for instance Chapter 3 of \cite{go96}), allow one to create very viable analogs of the classical
special functions of complex arithmetic. Thus, for instance, we are able to construct analytic characteristic-$p$ valued $L$-series and zeta functions which possess special values
extremely similar to those of classical $L$-series (such as an analog, due to Carlitz, of Euler's great formula for $\zeta_\Q(2m)$ for $m$ a positive
integer).  We also present here some new functions originally defined by B.\ Angl\`es (Subsection \ref{bruno}).

These analogies, along with others not mentioned here, encourage us to look for a deeper theory associated to these functions. In particular,
it is completely reasonable to search for a deeper theory of the zeroes of these functions and, as of this writing, such a theory is still very much
in its infancy. However, there have been a few notable victories in this regard. For instance, in the 1990's, based on earlier work of D.\ Wan, J.\ Sheats
\cite{Sh98} was able to establish that the zeroes of $\zeta_{\Fq[\theta]}(s)$ are simple and lie ``on the line'' $\Fq((1/\theta))$. This clearly is similar
to what is {\it expected} for the Riemann zeta function. 

Of course, basic to the study of the Riemann hypothesis, and, indeed, all complex valued $L$-series are their functional equations which are
symmetries under the transformation $s\mapsto m-s$ for some positive integer $m$. In finite characteristic it was realized very early on that
no such simple symmetry would be satisfied. However, in 1995 Dinesh Thakur \cite{Th95} published some calculations that led the present author
to construct a certain huge group $S_{(q)}$ (see Section \ref{sq} and \cite{Go11}) which {\it appears} to act as symmetries of the characteristic $p$
$L$-series (but in a way that is still far from clear). 

It is our purpose here to revisit $S_{(q)}$ as well as the zeroes in finite characteristic. We will establish the ubiquity of $S_{(q)}$, as well
as its naturality, by expressly presenting many different actions of $S_{(q)}$ on spaces (such as 
$\Zp$ or $\Fq((1/\theta))$),  functions on these spaces and measures on these spaces, see Section \ref{sq}. In some cases, these
actions are remarkably simple; see, for instance, Corollary \ref{sq6} or Proposition \ref{sq19}. Along the way, we obtain a new $p$-adic congruence
for binomial coefficients in Proposition \ref{sq10}.

It is well-known that every entire function over a nonarchimedean field is determined up to a constant by its zeroes. Using this, it was pointed
out in \cite{Go12} that the best one could hope for, at least in the case where the base ring $A=\Fq[\theta]$ is to have the zeroes lie in
a constant field extension of $\Fq((1/\theta))$. In other words, the zeroes should be {\it unramified}. As we point out in Subsection \ref{rh2}, 
if one views the classical Riemann hypothesis from the viewpoint of the Carlitz module, it also becomes a statement that the zeroes are
unramified!

We originally defined the action of $S_{(q)}$ on $\Fq((1/\theta))$. But as mentioned the zeroes will not always be in this field. They will, however,
always be algebraic over $\Fq((1/\theta))$ and so, with a choice of basis, we show how to further extend the action of $S_{(q)}$. It is expected that this action is indeed
{\it highly} dependent on the choice of basis chosen. However, when the extension is given by a constant field extension, the action is just the usual
extension by scalars (via the tensor product). In other words, the action is independent of the choice of basis and, in fact, this may ultimately give
an approach to showing the a zero is unramified.

In the final section of \cite{Con15}, A.\ Connes quotes Galois to support his contention that the classical Riemann hypothesis should somehow be the beginning and not the end of the understanding
of zeta zeroes. It is precisely in this optic that we view the potential action of $S_{(q)}$ (presented here and in \cite{Go12}) on the zeroes in characteristic $p$ as it seems to portend a much deeper understanding of these
zeroes. 

It is my great pleasure to thank Federico Pellarin and Rudy Perkins for many helpful comments on this paper.

\section{Review of basic Carlitz theory}\label{review}
We recall here some basic ideas of L.\ Carlitz that can be found in \cite{go96}.
\subsection{The Carlitz factorial}\label{fact}
As mentioned, we have $A=\Fq[\theta]$ where $q=p^{m_0}$ and $p$ is prime and $k=\Fq(\theta)$
its quotient field. 
\begin{defn}\label{def1}
Let $i$ be a positive integer. We set $[i]:=\theta^{q^i}-\theta$.
\end{defn}
\noindent
Clearly $\frac{d[i]}{d\theta}=-1$ and so $[i]$ is separable. Standard theory of finite fields then 
readily establishes
\begin{equation}\label{def2}
[i]=\prod_f f\,,
\end{equation}
where $f$ runs over the monic prime polynomials whose degree divides $i$.
\begin{defn}\label{def3}
1.\  We set $D_0:=1$, and for $i>0$, $D_i:=[i][i-1]^q\cdots [1]^{q^{i-1}}$.\\
2.\  We set $L_0:=1$, and for $i>0$, $L_i:=[i][i-1]\cdots [1]$\,.
\end{defn}

Notice that $L_i=[i]L_{i-1}$ and $D_i=[i]D_{i-1}^q$. The next result follows directly using Equation
(\ref{def2}).
\begin{prop}\label{def4}
We have $D_i$ is the product of all monic polynomials in $\theta$ of degree $i$. 
Moreover, $L_i$ is the least common multiple of all such monic polynomials.
\end{prop}

Let $j$ be a nonnegative integer written $q$-adically as $j=\sum_{t=0}^m c_t q^t$ where 
$0\leq c_t<q$ for all $t$. 
\begin{defn}\label{def5}
We set
\begin{equation}
\Pi(j):=\prod_{t=0}^m D_t^{c_t}\,.
\end{equation}
\end{defn}

\noindent
Notice, of course, that $\Pi(q^e)=D_e$. 
\begin{prop}\label{def6}
We have 
\begin{equation}\label{def7}
\Pi(q^j-1)=(D_0\cdots D_{j-1})^{q-1}=\frac{D_j}{L_j}=(-1)^j \prod_\alpha \alpha\,,
\end{equation}
where $\alpha$ runs over all nonzero polynomials in $A$ of degree $<j$.
\end{prop}
\begin{proof}
The first product is simply the definition of $\Pi(q^j-1)$. The second equality follows from
repeatedly using $D_i=[i]D_{i-1}^q$ and the equality follows upon expressing the product over
{\it all} nonzero elements of degree $<j$ as a corresponding product over the monic polynomials
of degree $<j$. \end{proof}
\subsubsection{A determinant formula}\label{det}
Here we will explain how the special factorials $\Pi(\frac{q^m-1}{q-1})$ can be obtained as a determinant
of Vandermonde type. We begin with a very clever observation of F. Voloch 
\cite{vo98} relating differentiation
and $q$-th powers. Let $f(\theta)\in A$ be a polynomial in $\theta$ and let $\partial_j=\partial_{\theta,j}$ be the $j$-th 
divided derivative in $\theta$ of $f(\theta)$. So one has $f(x)=\sum_j \partial_j f(\theta)(x-\theta)^j$ upon expanding
about $x=\theta$. 

\begin{prop}\label{det1}
With $f(\theta)$ as above,  and $i$ a nonnegative integer, we have
\begin{equation}\label{det2}
f^{q^i}(\theta)=\sum_j\partial_j f(\theta)[i]^j\,.
\end{equation}\end{prop}
\begin{proof}
One simply uses the above Taylor expansion, substitutes $\theta^{q^i}$ for $x$, and notes
that $f(\theta^{q^i})=f^{q^i}(\theta)$. \end{proof}

Now let $m$ be a fixed positive integer and let $\{x_0,\ldots, x_m\}\subset A_{<m+1}$. 
Set
\begin{equation}\label{det3}
M=M(x_0,\ldots, x_m):=\begin{pmatrix}
x_0&x_1&x_2&x_3&\dots&x_m\\
x_0^q&x_1^q&x_2^q&x_3^q&\dots &x_m^q\\
x_0^{q^2}&x_1^{q^2}&x_2^{q^2}&x_3^{q^3}&\dots &x_m^{q^2}\\
\vdots&\vdots& \vdots& \vdots&\vdots& \vdots\\
x_0^{q^m}&x_1^{q^m}&x_2^{q^m}&x_3^{q^m}&\dots &x_m^{q^m}\\
\end{pmatrix}\, ,
\end{equation}
\begin{equation}\label{det4}
V=V(0,[1], \ldots, [m]):=\begin{pmatrix}
1&0&0&0&\dots&0\\
1&[1]&[1]^2&[1]^3&\dots &[1]^m\\
1&[2]&[2]^2&[2]^3&\dots &[2]^m\\
\vdots&\vdots& \vdots& \vdots&\vdots& \vdots\\
1&[m]&[m]^2&[m]^3&\dots &[m]^m\\\end{pmatrix}\, ,\end{equation}
and,
\begin{equation}\label{det5}
W=W(x_0,\dots, x_m):=\begin{pmatrix}
x_0&x_1&x_2&x_3&\dots&x_m\\
\partial_1 x_0&\partial_1 x_1&\partial_1 x_2&\partial_1 x_3&\dots &\partial_1 x_m\\
\partial_2 x_0&\partial_2 x_1&\partial_2 x_2&\partial_2 x_3&\dots &\partial_2 x_m\\
\vdots&\vdots& \vdots& \vdots&\vdots& \vdots\\
\partial_m x_0&\partial_m x_1&\partial_m x_2&\partial_m x_3&\dots &\partial_m x_m\\
\end{pmatrix}\,.
\end{equation}

Notice, of course, that $M$ is an example of a Moore matrix, $V$ a matrix of Vandermonde type
and $W$ a Wronskian. As M.\ Papanikolas has communicated to the author \cite{po13}, Proposition \ref{det1} immediately
implies the next result relating these three famous types of matrices.

\begin{prop}\label{det6}
With the above notations, we have
\begin{equation}\label{det7}
M=VW\,.
\end{equation}
\end{prop}

Now choose $x_j=\theta^j,$ for $j=0,\dots, m$. 

\begin{cor}\label{det8}
We have 
\begin{equation}\label{det9}
\Pi\left(\frac{q^m-1}{q-1}\right)=\det V(0,[1], \ldots, [m])\,.\end{equation}
\end{cor}
\begin{proof}
Note that the Wronskian for our choice of functions is now unipotent. So from Proposition \ref{det6}
we obtain, with the obvious notation, $\det M=\det V$. But Moore's result (Section 1.3 of 
\cite{go96}) now establishes that
$\Pi(\frac{q^m-1}{q-1})=\det M(1,\theta,\dots, \theta^m)$. \end{proof}

\begin{rems}\label{det10}
Corollary \ref{det8} is quite remarkable in that it shows that for arbitrary choices of $\{x_i\}$ the
Wronskian determinant and the Moore determinant always differ by $\Pi((q^m-1)/(q-1))$. Moreover,
we also deduce that our matrix $V$ is invertible. Therefore, and more remarkably, we see that one can always
express the divided derivatives of elements in $A_{<m+1}$ in terms of their $q^j$-th powers with
denominator at most $\Pi((q^m-1)/(q-1))$. 
\end{rems}

\subsection{The Carlitz polynomials}\label{carpoly}
Classically, one passes from $n!$ to the binomial coefficient functions $\binom{s}{n}$ defined, as
usual, by $\binom{s}{n}:=\frac{s(s-1)\cdots(s-n+1)}{n\cdot (n-1)\cdots 1}$. Clearly 
$\binom{s}{n}$ is a polynomial of degree $n$ in $s$ and is well-known to take on integer values
for integer arguments. 

As $\binom{s}{n}$ has degree $n$, the collection $\{\binom{s}{n}\}$ forms a basis for the vector
space of polynomials in $s$ over any field of characteristic $0$. Let $p(s)$ be a polynomial of 
degree $d$ which we write as $p(s)=\sum_{k=0}^d a_{p,k}\binom{s}{k}$. It is further well known that
\begin{equation}\label{def8}
a_{p,k}=\sum_{j=0}^k (-1)^{k-j} \binom{k}{j}p(j)\,.
\end{equation}

The Carlitz polynomials, recalled here, play the role for the Carlitz factorial that is played by
$\binom{s}{n}$ for $n!$. It will be convenient to begin with some definitions. 
\begin{defn}\label{def9}
1. Let $A_d$ be the polynomials in $A$ of degree $d$.\\
2. Let $A_+$ be the monic polynomials of $A$ and $A_{d,+}$ the monic polynomials in $A$ of degree $d$.\\
3. Let $A_{<d}$ be the polynomials in $A$ of degree strictly less than $d$.
\end{defn}
\begin{defn}\label{def10}
We set $e_0(x):=x$ and, for an integer $t>0$, 
\begin{equation}\label{def11}
e_t(x):=\prod_{\alpha\in A_{<t}}(x-\alpha)\,.
\end{equation}
\end{defn}
Since the roots of $e_j(x)$ are a finite dimensional $\Fq$-vector space of dimension $t$,  and it is
separable, it is well-known that is therefore an $\Fq$-linear function of degree $q^t$ in $x$. 

\begin{prop}\label{def11.1}
Let $h$ be any monic polynomial of degree $t$. Then $e_t(h)=D_t$.
\end{prop}
\begin{proof}
This follows immediately from Equation (\ref{def11}).\end{proof}

Let $j$ be a nonnegative integer written $q$-adically as $j=\sum_{t=0}^m c_t q^t$, as above.
\begin{defn}\label{def12}
1. We set $g_j(x):=\prod_t e_t(x)^{c_t}$. \\
2. We set 
\begin{equation}\label{def13}
G_j(x):=\prod_t \left(\frac{e_t(x)}{D_t}\right)^{c_t}=\frac{g_j(x)}{\Pi(j)}\,.\end{equation}
\end{defn}

Notice that both $g_j(x)$ and $G_j(x)$ have degree $j$ in $x$. Moreover, if $\zeta\in \Fq$, then
$g_j(\zeta x)=\zeta^jg_j(x)$ and $G_j(\zeta x)=\zeta^jG_j(x)$ for all $j$.

\begin{prop}\label{def13.1}
Let $\alpha\in A$. Then $G_j(\alpha)\in A$ for all $\alpha\in A$.
\end{prop}
\begin{proof}
The values $\{G_{q^t}(\alpha)=\frac{e_t(\alpha)}{D_t}\}$ occur as the coefficients
of the Carlitz module $C_\alpha(\tau)$ where $\tau$ is the $q$-th power mapping; as such they
are in $A$. Consequently, $G_j(\alpha)\in A$ for all $\alpha \in A$ also.\end{proof}

Let $y$ be another indeterminate. With $j$ as above, we then have
\begin{equation}\label{def14}  
(x+y)^j=\sum_{w+v=j}\binom{j}{v}x^vy^w=(x+y)^{\sum c_t q^t}=\prod_t(x^{q^t}+y^{q^t})^{c_t}\,.
\end{equation}
Notice, obviously, that $e_t(x+y)=e_t(x)+e_t(y)$. The {\it addition formulas} now follow from Equation (\ref{def14}) and are given by our next result.
\begin{prop}\label{def15}
We have:
 $$g_j(x+y)=\sum_{w+v=j}\binom{j}{v}g_v(x)g_w(y)\,,$$
 and
$$G_j(x+y)=\sum_{w+v=j}\binom{j}{v}G_v(x)G_w(y)\,.$$
\end{prop}
\noindent
In other words, both $\{g_j(x)\}$ and $\{G_j(x)\}$ satisfy the binomial theorem. We call them
``the Carlitz polynomials," and we will now define their duals.
\begin{defn}\label{def16}
1. 
Let $0\leq v<q$ and $t\geq 0$. We set
\begin{equation}\label{def17}
\hat g_{v q^t}(x):=\begin{cases} e_t(x)^v & \mbox{if } v<q-1\\ e_t(x)^{q-1}-D_t^{q-1}& \mbox{if }v=q-1\,.
\end{cases}\end{equation}
We set $\hat{G}_{v q^t}(x):=\frac{\hat g_{v q^t}(x)}{D_t^v}$ for all $v$ and $t$. \\
2. Now let $c=\sum c_t q^t$ with $0\leq c_t<q$ for all $t$. We set $\hat g_c(x):
=\prod_t \hat g_{c_tq^t}(x)$ and $\hat G_c(x):= \prod_t \hat{G}_{c_t q^t}(x)$.
\end{defn}
\noindent

\begin{rems}\label{def17.1}
Note  that $G_j(x)=\frac{g_j(x)}{\Pi(j)}$ and $\hat G_j(x)=\frac{\hat g_j(x)}{\Pi(j)}$
for all $j$.
 Note also that both $\hat{g}_{(q-1)q^t}(x)$ and $\hat G_{(q-1)q^t}(x)$ vanish on all
polynomials of degree $t$.\end{rems}

\begin{prop}\label{def18}
Let $j=\sum c_tq^t$ be as above.\\
1. Let $\zeta\in \Fq$. Then $\hat g_j(\zeta x)=\zeta^j \hat g_j(x)$ and $\hat{G}_j(\zeta x)=\zeta^j \hat
G_j(x)$.\\
2. Let $\alpha \in A$. Then $\hat G_j(\alpha)\in A$ also.\\
3. Let $m$ be a positive integer. Then 
\begin{equation}\label{def19} 
\frac{e_m(x)}{x}=\hat g_{q^m-1}(x)\,.\end{equation}\end{prop}
\begin{proof} Part 1 follows directly from the definitions. Part 2 follows from Proposition \ref{def13.1}.
Finally, by the remark just above, one sees that both $e_m(x)/x$ and $\hat g_{q^m-1}(x)$ are
monic of the same degree and with the same zeroes; thus they are equal.\end{proof}

\begin{prop}\label{def20}
We have for $j\geq 0$
\begin{equation}\label{def21}
\hat g_j(x+y)=\sum_{e+f=j} \binom{j}{e} g_e(x)\hat g_f(y)=\sum_{e+f=j}\binom{j}{e}\hat g_e(x) g_f(y)\,,
\end{equation}
and
\begin{equation}\label{def22}
\hat G_j (x+y)=\sum_{e+f=j} \binom{j}{e}G_e(x)\hat G_f(y)=\sum_{e+f=j}\binom{j}{e} \hat G_e(x)G_f(y)\,.
\end{equation}
\end{prop}
\begin{proof} The proof of Proposition \ref{def15} works here also.\end{proof}

The following simple lemma is, of course, very well-known.

\begin{lemma}\label{def22.1}
Let $m$ be a positive integer. Then for any $0\leq j\leq q^m-1$, we have
$$\binom{q^m-1}{j}\equiv (-1)^j \pmod{p}\,.$$
\end{lemma}
\begin{proof} Note that in characteristic $p$, $(1+z)^{q^m-1}=(1+z)^{q^m}(1+z)^{-1}=
(1+z^{q^m})(1+z)^{-1}$. Now use the geometric series to expand $(1+z)^{-1}$. \end{proof}
As an immediate corollary, we obtain the next result.
\begin{cor}\label{def22.2}
Let $m$ be a positive integer. We then have:\\
1. \begin{equation}\label{def23}
g_{q^m-1}(x+y)=\sum_{e+f=q^m-1} (-1)^e g_e(x)g_f(y)\,.
\end{equation}
\noindent
2. \begin{equation}\label{def24}
g_{q^m-1}(x-y)=\sum_{e+f=q^m-1}g_e(x)g_f(y)\,.\end{equation}
\noindent
3. \begin{equation}\label{def25}
G_{q^m-1}(x+y)=\sum_{e+f=q^m-1} (-1)^e G_e(x)G_f(y)\,.\end{equation}
\noindent
4. \begin{equation}\label{def26}
G_{q^m-1}(x-y)=\sum_{e+f=q^m-1}  G_e(x)G_f(y)\,.
\end{equation}\end{cor}

\begin{proof} Recall that $g_j(-y)=(-1)^j g_j(y)$ and $G_j(-y)=(-1)^j G_j(y)$ and the result follows
immediately.\end{proof}
The same argument immediately establishes the next essential result.
\begin{cor}\label{def27}
Let $m$ be a positive integer. Then
\begin{equation}\label{def28}
\hat g_{q^m-1}(x-y)=\sum_{e+f=q^m-1}g_e(x)\hat g_f(y)=\sum_{e+f=q^m-1}\hat g_e(x)g_f(y)\,.
\end{equation}\end{cor}

Let $L$ be any field containing $A$ and let $f(x)\in L[x]$ be a polynomial of degree $d$. 
Since $g_j(x)$ and $\hat g_j(x)$ have degree $j$ in $x$, we deduce two expressions for
$f(x)$ in terms of these polynomials:
\begin{equation}\label{def29}
f(x)=\sum_{i=0}^d a_{f,i} g_i(x)=\sum_{i=0}^d \hat a_{f,i}\hat g_i(x)\,.
\end{equation}
We will now deduce two ``integrals'' (actually finite sums) for the coefficients $\{a_i\}$ and
$\{\hat a_i(x)\}$. 

Choose $m$ such that $q^m>d$. 
\begin{theorem}\label{def30}
We have
\begin{equation}\label{def31}
(-1)^m\frac{D_m}{L_m} f(x)=\sum_{\alpha\in A_{<m}}f(\alpha) \frac{e_m(x)}{x-\alpha}\,.
\end{equation}
\end{theorem}
\begin{proof}
Let $h(x)$ be the sum on the right hand side of Equation \ref{def31}.
Note that since $e_m(\alpha)=0$ for $\alpha\in A_{<m}$, we deduce $e_m(x)/(x-\alpha)$  is a polynomial in $x$ of 
degree $q^m-1$. Thus $h(x)$ has degree in $x$ at most $q^m-1$. Moreover, for 
$\alpha\in A_{<m}$ we have $h(\alpha)=f(\alpha)\frac{e_m(x)}{x-\alpha}\vert_{x=\alpha}$.

But 
\begin{equation}\label{def32}
\frac{e_m(x)}{x-\alpha}\vert_{x=\alpha}=\frac{e_m(x-\alpha)}{x-\alpha}\vert_{x=\alpha}=
\frac{e_m(x)}{x}\vert_{x=0}\,.\end{equation}
Moreover by Proposition \ref{def6}, $\frac{e_m(x)}{x}\vert_{x=0}=(-1)^m\frac{D_m}{L_m}$. Now both
$h(x)$ and $(-1)^m\frac{D_m}{L_m} f(x)$ are polynomials of degree at most $q^m-1$ that agree
on the $q^m$ points in $A_{<m}$; thus they are equal.
\end{proof}

\begin{theorem}\label{def33} Let $f(x)$ and $\{a_{f,i}\}$ be as in Equation (\ref{def29}). Let
$q^m>d$ where $d$ is the degree of $f(x)$. Then 
we have 
\begin{equation}\label{def34}
(-1)^m\frac{D_m}{L_m} a_{f,i}=\sum_{\alpha\in A_{<m}}f(\alpha)\hat{g}_{q^m-1-i}(\alpha)\,.
\end{equation}
\end{theorem}
\begin{proof}
Theorem \ref{def30} assures us that 
\begin{equation}\label{def35}
(-1)^m \frac{D_m}{L_m}f(x)=\sum_{\alpha\in A_{<m}}f(\alpha)\frac{e_m(x)}{x-\alpha}=
\sum_{\alpha\in A_{<m}}f(\alpha)\frac{e_m(x-\alpha)}{x-\alpha}\,.
\end{equation}
By the third part of Proposition \ref{def18} we can rewrite this last sum as
$$\sum_{\alpha\in A_{<m}} f(\alpha)\hat g_{q^m-1}(x-\alpha)\,,$$
But by Corollary \ref{def27}, we have 
$$\hat g_{q^m-1}(x-\alpha)=\sum_{e+f=q^m-1}g_e(x)\hat g_f(\alpha)\,.$$
The result follows upon collecting terms.
\end{proof}
The same proof also gives the following {\it dual} result.
\begin{theorem}\label{def36}
With the hypotheses of Theorem \ref{def33} we
have 
\begin{equation}\label{def37}
(-1)^m\frac{D_m}{L_m} \hat a_{f,i}=\sum_{\alpha\in A_{<m}}f(\alpha){g}_{q^m-1-i}(\alpha)\,.
\end{equation}
\end{theorem}
Our next result presents the compatibility between the various expressions for the coefficients
given in Theorem \ref{def33}.
\begin{prop}\label{def38} 
Let $q^{m_1}>i$ and $q^{m_2}>i$. Then
\begin{equation}\label{def39}
(-1)^{m_1}\frac{L_{m_1}}{D_{m_1}}
\sum_{\alpha\in A_{<m_1}}\hat g_{q^{m_1}-1-i}(\alpha)f(\alpha)=
(-1)^{m_2}\frac{L_{m_2}}{D_{m_2}}\sum_{\alpha \in A_{<m_2}}\hat g_{q^{m_2}-1-i}(\alpha)f(\alpha)\,.
\end{equation}\end{prop}
\begin{proof}
Write $i$ $q$-adically as $\sum_{t=0}^e c_tq^t$ with $c_e\neq 0$. We will show that both sides 
of Equation (\ref{def39}) are equal to 
$$(-1)^{e+1}\frac{L_{e+1}}{D_{e+1}}\sum_{\alpha \in A_{<e+1}}\hat{g}_{q^{e+1}-1-i}(\alpha)f(\alpha)\,.$$
\noindent
Without loss of generality, set $m=m_1$.
Now, by definition, $\hat g_{q^m-1-i}(\alpha)=0$ for all $\alpha$ of degree $>e$ and $<m$. 
Thus 
\begin{equation}\label{def40}
(-1)^m\frac{L_m}{D_m}\sum_{\alpha\in A_{<m}}\hat g_{q^m-1-i}(\alpha)f(\alpha)=
(-1)^m \frac{L_m}{D_m}\sum_{\alpha \in A_{<e+1}}\hat g_{q^m-1-i}(\alpha)f(\alpha)\,.
\end{equation}
Now notice that, by definition, $\hat g_{q^m-1-i}(\alpha)=(-1)^{m-e-1}(D_{e+1}\cdots
D_{m-1})^{q-1}\hat g_{q^{e+1}-1-i}(\alpha)$. The result follows upon noticing that 
$\Pi(q^m-1)=\frac{D_m}{L_m}$.

\end{proof}

\begin{rems}\label{def40.1} The above compatibility formula, Proposition \ref{def38}, has some 
important corollaries. In particular, we can give formulae for the coefficients $\{a_{f,j}\}$ which are independent of $d$. Indeed, let $f(x)$ be a polynomial of degree $d$,  and write $f(x)=\sum_{i=1}^d
a_{f,i}g_i(x)$, also as before. Let $i$ be a fixed integer $\leq d$ which we write
$q$-adically, as $\sum_{t=0}^ec_tq^t$ with $c_e\neq 0$ as above. Notice then the 
values of $f(x)$ on the points $x\in A_{<e}$ depends {\em only} on the truncation of the expansion of
$f(x)$ given by $\sum_{i=0}^{q^{e+1}-1}a_{f,i}g_i(x)$ simply because the other elements in the sum
vanish on these values. Proposition \ref{def38} implies that $a_{f,i}=(-1)^{e+1} \frac{L_{e+1}}{D_{e+1}}\sum_
{\alpha\in A_{<e+1} }f(\alpha)\hat g_{q^{e+1}-1-i}(\alpha)$ which does {\bf not} depend on
$d$ in any fashion. It does not appear possible to present similar formulae for the  dual coefficients.
\end{rems}

Next we present a formula for the coefficients $\{a_{f,i}\}$ and $\{\hat a_{f,i}\}$ that involves
summing over monics of a given degree. We begin with an analog of Theorem \ref{def30}.
As before, let $f(x)$ be a polynomial of degree $d$ and let $m$ be chosen so that 
$q^m>d$.
\begin{theorem}\label{def41}
With the above hypotheses, we have 
\begin{equation}\label{def42}
(-1)^m \frac{D_m}{L_m}f(x)=\sum_{h\in A_{m,+}} f(h)\frac{e_m(x)-D_m}{x-h}\,.
\end{equation}
\end{theorem}
\begin{proof}
The sum on the right of Equation (\ref{def42}) is a polynomial of degree at most $q^m-1$. As before
we see that it, and $(-1)^m f(x)$, have the same values at the points in $A_{m,+}$. Thus they
must be equal.\end{proof}

Let $f(x)$ of degree $d$ be written in terms of the $g_j(x)$ and their duals as in Equation (\ref{def29}) and let $q^m>d$.

\begin{theorem}\label{def43}
We have 
\begin{equation}\label{def44}
(-1)^m \frac{D_m}{L_m} a_{f,i}=\sum_{h\in A_{m,+}}f(h) \hat{g}_{q^m-1-i}(h)\,,
\end{equation}
and
\begin{equation}\label{def45}
(-1)^m \frac{D_m}{L_m} \hat a_{f,i}=\sum_{h\in A_{m,+}} f(h)g_{q^m-1-i}(h)\,.
\end{equation}
\end{theorem}
\begin{proof} 
The result follows as before upon noting that, for $h\in A_{m,+}$, we have
$e_m(x)-D_m=e_m(x-h)$.\end{proof}


Expanding the Carlitz polynomial $g_j(x)$ leads to the following {\it orthogonality formulae}.

\begin{theorem}\label{def46}
1. For $l<q^m$ and $j$ arbitrary 
\begin{equation}\label{def47}
\sum_{\alpha \in A_{<m}}\hat g_l(\alpha)g_j(\alpha)
=\begin{cases}0 &\mbox{if } l+j\neq q^m-1\\(-1)^m\frac{D_m}{L_m}& \mbox{if } l+j=q^m-1\,.
\end{cases}
\end{equation}
\noindent
2. For $l<q^m,$ $j<q^m$
\begin{equation}\label{def48}
\sum_{h \in A_{m,+}}\hat g_l(h)g_j(h)
=\begin{cases}0 &\mbox{if } l+j\neq q^m-1\\(-1)^m\frac{D_m}{L_m}& \mbox{if } l+j=q^m-1\,.
\end{cases}
\end{equation}
\end{theorem}

We now pass to the corresponding formulae for $G_i(x)$ and $\hat G_i(x)$ which are, in fact,
simpler than those given for $g_i(x)$ and $\hat g_i(x)$.

Recall that in Equation (\ref{def29}) we expressed a polynomial $f(x)$ in terms of both
$\{g_i(x)\}$ and $\{\hat g_i(x)\}$ as
$$f(x)=\sum_{i=0}^d a_{f,i} g_i(x)=\sum_{i=0}^d \hat a_{f,i}\hat g_i(x)\,.$$
We now write
\begin{equation}\label{def49}
f(x)=\sum_{i=0}^d A_{f,i} G_i(x)=\sum_{i=0}^d \hat A_{f,i}\hat G_i(x)\,.
\end{equation}
\noindent
From the definitions, we immediately conclude that
$$\Pi(i)a_{f,i}=A_{f,i} \quad {\rm and} \quad \Pi(i)\hat a_{f,i}=\hat A_{f,i}\,.$$
Now recall that by Proposition \ref{def6} we have $\frac{D_m}{L_m}=\Pi(q^m-1)$. But notice
$\Pi(q^m-1)=\Pi(i)\Pi(q^m-1-i)$. Thus from Theorem \ref{def33}, Theorem \ref{def36}, and
Theorem \ref{def43}, we obtain the next fundamental theorem.

\begin{theorem}\label{def50}
With the hypotheses of Theorems \ref{def33}, \ref{def36} and \ref{def43} we have:\\
1. \begin{equation}\label{def51}
(-1)^m A_{f,i}=\sum_{\alpha \in A_{<m}} f(\alpha) \hat G_{q^m-1-i}(\alpha)\,.
\end{equation}
\noindent
2. \begin{equation}\label{def52}
(-1)^m\hat A_{f,i}=\sum_{\alpha\in A_{<m}} f(\alpha) G_{q^m-1-i}(\alpha)\,.\end{equation}
\noindent
3. \begin{equation}\label{def53}
(-1)^m A_{f,i}=\sum_{h\in A_{m,+}} f(h) \hat G_{q^m-i-1}(h)\,.\end{equation}
\noindent
4. \begin{equation}\label{def54}
(-1)^m \hat A_{f,i}=\sum_{h\in A_{m,+}} f(h)G_{q^m-1-i}(h)\,\end{equation}
\end{theorem}

\begin{cor}\label{def55}
Both $\{G_i(x)\}$ and $\{\hat G_i(x)\}$ are bases for the $A$-module of elements $f(x)\in
k[x]$ such that $f(\alpha)\in A$ for all $\alpha\in A$.\end{cor}
\begin{proof}
Just express $f(x)$ using the theorem.\end{proof}
 
We leave the obvious orthogonality formulae for $G_i(x), \hat{G}_i(x)$, which follow immediately from Theorem \ref{def46}
to the reader.

\section{$\mathfrak v$-adic continuous functions}\label{vadic}
\subsection{Basic notions}\label{vadic.5}
We begin with some notation that we will use throughout the paper. Let $X,Y$ be two
topological spaces. Then we denote the set of continuous functions $f\colon X\to Y$ by
$\mathcal C(X,Y)$.

The binomial theorem implies that the binomial polynomials $\{\binom{s}{j}\}$ take the nonnegative 
integers to themselves as was mentioned above. Let $p$ be a prime number.  Since the
$p$-adic integers, $\Zp$, form a compact (and therefore closed) algebra, continuity implies that
the binomial polynomials also map $\Zp$ into itself. It is a fundamental result, due to K.\  Mahler,
\cite{Ma58}
that the set $\{\binom{s}{j}\}$ forms an {\it orthogonal basis} for the nonarchimedean Banach space
$\mathcal C(\Zp,\Qp)$. In this section we will briefly recall the theorem of 
C. G. Wagner \cite{Wa70} which establishes a similar result for the Carlitz polynomials. 

In keeping with the theme of this paper, we will sketch a ``modern'' proof of this result
following the excellent paper \cite{Co00} by K.\ Conrad. We begin by recalling a well known result of J.-P. Serre \cite{Se62}.
Let $\{K, \vert \cdot \vert\}$ be a local field with ring of integers $O$ (i.e., a finite extension of
$\Qp$ for some prime $p$ or a formal Laurent series field over a finite field), and
maximal ideal $M$. Let $(E,\Vert\cdot \Vert )$ be a Banach space
over $K$ and  we assume that $\Vert \cdot \Vert$ and $\vert \cdot \vert$ have the same
value group. We set $E_0:=\{x\in E\colon \Vert x\Vert \leq 1\}$, which is an $O$-module, and 
set $\bar E:=E_0/ME_0$. Clearly $\bar E$ is a vector space over the field $\F:=O/M$. 

\begin{examples}\label{vadic1}
a. Let $E_c={\mathcal C}(O,K)$ equipped with the norm $\Vert f\Vert:=\sup_{x\in O}\{\vert f(x)\vert\}$. Clearly
both $E_c$ and $K$ have the same value groups. Moreover, $E_{c,0}={\mathcal C}(O,O)\subset {\mathcal C}(O,K)$ and thus $\bar E_c\simeq {\mathcal C}(O,\F)$ where $\F$ has the discrete topology.\\
b. Let $K$ have positive characteristic and let $\F$ be its field of constants (and so we identify
$\F$ and $O/M$ in the usual fashion) where the cardinality of $\F$ is $q$. Let $E_l:=\Hom_\F (O,K)$ be the $\F$-vector space of $\F$-linear,
continuous functions from $O$ to $K$ equipped with the sup norm as in Part a. Clearly 
$E_l$ is a closed sub Banach space of $E_c$ and one sees readily that $\bar E_l=\Hom_\F(O,\F)$, the vector
space of continuous $\F$-linear functions from $O$ to $\F$ with $\F$ again having the discrete topology. 
\end{examples}

As usual, given $E$, as above, and 
a sequence $\mathfrak E:=\{\mathfrak e_0, \mathfrak e_1, \ldots\}$ of elements of $E$,  we say that
$\mathfrak E$ is an {\it orthonormal basis} for $E$ if and only if every $x\in E$ can be written
$x=\sum_{i=0}^\infty c_i \mathfrak e_i $ with $\{c_i\}\subset K$, $c_i\to 0$ as $i\to \infty$ and $\Vert x\Vert=\sup_i \{\vert c_i\vert\}$. If $\mathfrak E$ is an orthonormal basis, a moment's reflection assures one that the above expansion for $x$ is unique.

The following basic and well-known result is due to J.-P.\ Serre \cite{Se62}.
\begin{lemma}\label{vadic2}
Let $\{K, E,\Vert\cdot\Vert\}$ be as in Part a of Examples \ref{vadic1}. Let $\mathfrak E=
\{\mathfrak e_0,\mathfrak e_1,\ldots\}$ be a sequence of elements in $E$. Then a necessary and sufficient condition for $\mathfrak E$
to be an orthonormal basis for $E$ is that $\mathfrak e_i\in E_0$ for all $i$ and the reductions $\{\bar {\mathfrak e}_i\}\subset \bar E$ form a basis for $\bar E$ as an $\F$-vector space.\end{lemma}

The import of Lemma \ref{vadic2} is that it allows us to find orthonormal bases  for a Banach space
$E$ by finding vector space bases of $\bar E$. 

\subsection{The basic construction}\label{vadic2.5}
Let $K$ be our local field of characteristic $p>0$ and let $E_c={\mathcal C}(O,K)$ all as in Examples \ref{vadic1}.
Let $q=p^{n_0}$ be the order of $\F$ and  
let $\mathfrak E=\{\mathfrak e_0,\mathfrak e_1\dots\}$ be an orthonormal basis for $E_l$. Our next definition is based on the construction 
of the Carlitz polynomials given in Definition \ref{def12}.
\begin{defn}\label{vadic3}
Let $j$ be a nonnegative integer written $q$-adically as $j=\sum_{t=0}^\alpha c_tq^t$ with
$0\leq c_t<q$ all $t$. We set $\mathfrak g_j:=\prod_{t=0}^\alpha \mathfrak e_t^{c_t}$ and we put $\mathfrak G=\{\mathfrak g_j\}$.
\end{defn}
\noindent
We say that $\mathfrak G$ is {\it the extension of $\mathfrak E$ via $q$-adic digits}.

The next basic result is due to Conrad \cite{Co00}.

\begin{theorem}\label{vadic4}
With the above definitions, $\mathfrak G$ is an orthonormal basis for $E_c={\mathcal C}(O,K)$. 
\end{theorem}
\begin{proof} (Sketch) Note first that by continuity $ {\mathcal C}(O,\F)=\lim\limits_{\longrightarrow} \Maps (O/M^j,\F)$ and $ \Hom_\F (O,\F)=\lim\limits_{\longrightarrow}\Hom_\F(O/M^j,\F)$ where $O/M^j$ has the 
discrete topology.

Now let $\mathfrak E=\{\mathfrak e_0,\mathfrak e_1,\dots\}$ be our orthonormal basis for
$\Hom_F(O, K)$ with reductions $\bar {\mathfrak e}_i$ and set 
$H_i=\cap_{c=0}^i \ker \bar{ \mathfrak e}_c$; one sees readily that $H_i$ has codimension
$i$. Since the $\{H_i\}$ give a basis for the topology at the origin in $O$,  we further deduce that ${\mathcal C}(O,\F)=
\lim\limits_{\longrightarrow} \Maps (O/H_i,\F)$ and $\Hom_\F(O,\F)=
\lim\limits_{\longrightarrow} \Hom_\F(O/H_i, \F).$ 

We are thus reduced to the following combinatorial problem: Let $V$ be a vector space over 
$\F$ of dimension $n$
 and let $\phi_i$ be a basis for the dual space of $V$. Let $\Phi$ be the
set of functions on $V$ created out of $\{\phi_i\}$ via digit expansions (as above) for $0<j<q^n-1$. We  need to
show that $\Phi$ is then an $\F$-basis for $\Maps (V,\F)$. It is sufficient to show that $\Phi$ spans
$\Maps(V,\F)$. To see this, let $v\in V$. Then Conrad explicitly constructs an element $h_v$ in $\Span(\Phi)$ with the property that $h_v(v)=1$ but $h_v(w)=0$ for $w\neq v$. These elements clearly
span all maps from $V$ to $\F$ and the proof is complete.
\end{proof}

We also have the following generalization of the above theorem which is established in a 
similar fashion.

\begin{theorem}\label{vadic5}
Let $K$, $\F$ be as above and let $\F_0\subseteq \F$ be a subfield of cardinality $q_0$ (so that
$q=q_0^{n_0}$). Let $\mathfrak E=\{\mathfrak e_0,\mathfrak e_1,\dots\}$ be an orthonormal
basis of $\Hom_{\F_0}(O,K)$ such that $\cap_{j=0}^{in_0-1} \ker \bar{\mathfrak e}_j$ has
codimension $in_0$ for all $i\geq 1$. Let $\mathfrak G$ be the extension of $\mathfrak E$ using
$q_0$-adic digits. Then $\mathfrak G$ is an orthonormal basic for ${\mathcal C}(O,K)$.\end{theorem}
\subsection{Applications to $v$-adic continuous functions}\label{vadic6}
We now return to the case of the Carlitz polynomials of Subsection \ref{carpoly} and use the results
of the previous subsection to obtain orthonormal bases.

As in Subsection \ref{carpoly} we let $A=\Fq[\theta]$ and $k=\Fq(\theta)$. Let $K=k_{(\theta)}$ be
the completion of $k$ at the prime $(\theta)$ and let $O\subset K=\Fq[[\theta]]$ be the ring of
integers. Let $\{e_t(x)\}$ be as in Definition \ref{def10} and set $\mathfrak E:=\{e_t(x)/D_t\}$.

\begin{prop}\label{vadic7}  The set $\mathfrak E$ forms an orthonormal basis for $\Hom_{\Fq}
(O,K)\,.$
\end{prop}
\begin{proof}
By Lemma \ref{vadic2}, it suffices to show that the reductions $\{\overline{e_t(x)/D_t}\}$ form
a vector space basis of the continuous $\Fq$-linear maps from $O$ to $\Fq$ equipped with the 
discrete topology. Notice that $e_t(x)/D_t$ vanishes at $\theta^i$ for $i<t$ and $e_t(\theta^t)/D_t=1$.
Thus $\{e_i(x)/D_i\}_{i=0}^{t-1}$ forms a basis for the dual of $\Fq[[\theta]]/(\theta^t)$ and the result is established.
\end{proof}

Let $\mathfrak G:=\{G_j(x)\}$ where $G_j(x)$ is the Carlitz polynomial of Definition \ref{def12}.

\begin{cor}\label{vadic8}
The Carlitz polynomials $\mathfrak G$ form an orthonormal
basis for ${\mathcal C}(O,K)$.
\end{cor}
\begin{proof}
This is exactly the content of Theorem \ref{vadic4}.\end{proof}

Let $\mathfrak v$ now be any prime of $A$ with completions $O:=A_\mathfrak v$ 
and $K=k_\mathfrak v$. 
Very similar arguments give the following essential result.
\begin{theorem}\label{vadic9} The sequence $\mathfrak E$ forms an orthonormal basis for
$\Hom_\Fq(O,K)$. The sequence $\mathfrak G$ forms an orthonormal basis for ${\mathcal C}(O,K)$.
\end{theorem}
Theorem \ref{vadic9} is the direct analog of Mahler's theorem for binomial coefficients.

Let $\hat{\mathfrak G}$ be the sequence of dual Carlitz polynomials of Definition \ref{def16}.
\begin{cor}\label{vadic10}
The sequence $\hat{\mathfrak G}$ is an orthonormal basis for ${\mathcal C}(O,K)$.
\end{cor}
\begin{proof}
Notice that $G_t(x)$ and $\hat G_t(x)$ have the same degree and both sets of polynomials are
bases for the space of $A$-valued polynomials with coefficients in $k$. Thus the result follows
from Lemma \ref{vadic2}.\end{proof}

Let $f\in {\mathcal C}(O,K)$. By Corollaries \ref{vadic8} and \ref{vadic10} we can write
\begin{equation}\label{vadic11}
f(x)=\sum_{i=0}^\infty A_{f,i}G_i(x)=\sum_{i=0}^\infty \hat{A}_{f,i}\hat{G}_i(x)\,,
\end{equation}
as in Equation \ref{def49} with the coefficients tending to $0$ as $i$ goes to infinity.

Write $i$ $q$-adically as $\sum_{t=0}^e c_tq^e$ with $0\leq c_t<q$ all $t$ {\it and} $c_e\neq 0$.
Then the discussion of Remarks \ref{def40.1} immediately gives the next result.

\begin{prop}\label{vadic12}
We have 
\begin{equation}\label{vadic13}
(-1)^{e+1} A_{f,i}=\sum_{\alpha\in A_{<e+1}}f(\alpha)\hat G_{q^{e+1}-1-i} (\alpha)\,.
\end{equation}
\end{prop}
\noindent
As remarked before there is no obvious analog of Proposition \ref{vadic12} for the dual coefficients.
\section{Measure theory}\label{meas}
The previous section has equipped us with a good description of continuous functions on the
completions of $A$. Here we will use that to give a dual description of the algebra of measures
on these completions. We begin by recalling the corresponding theory of $p$-adic measures
on subsets of $\Zp$..
\subsection{Distributions and measures}\label{meas1}
Let $X$ be any compact open subset of the $p$-adic numbers such as $\Zp$ or $\Zp^\ast$.
\begin{defn}\label{meas2}
A $p$-adic distribution $\nu$ on $X$ is a finitely additive $\Qp$-valued function on the 
compact-open subsets of $X$.
\end{defn}
\noindent
More precisely, let $\{U_1,\ldots, U_m\}$ be a collection of pair-wise disjoint compact-open
subsets of $X$. Then we have
\begin{equation}\label{meas3}
\mu\left(\bigcup_{j=1}^mU_j\right)=\sum_{j=1}^m \mu (U_j)\,.
\end{equation}
\begin{defn} Let $\mu$ be a $p$-adic distribution. If the values $\mu(U)\in \Qp$ are 
bounded for all compact-open subsets $U\subset X$ (i.e., all such values lie in
a compact subset of $\Qp$), then we say that $\mu$ is a 
{\it $p$-adic measure}.
\end{defn}

\begin{example}\label{meas4}
Let $\alpha\in X$ and let $U\subseteq X$ be a compact-open. We define 
\begin{equation}\label{meas6}
\delta_\alpha(U)=\begin{cases} 1, &\text{if } \alpha\in U\\ 
                                                   0, &\text{otherwise} \end{cases}\,.
                                                   \end{equation}
Clearly $\delta_\alpha$ is a measure called the {\it Dirac measure at $\alpha$}.
\end{example}                                                   
\subsection{Integration}\label{meas7}
Let $\mu$ be a measure on $X$ as in our previous subsection and let $f\colon X\to \Qp$ be
a continuous function.  For each $a+p^j\Zp\subseteq X$ choose a point $x_{a,j}\in a+p^j\Zp$
and let $X_j:=\{x_{a,j}\}$. We can then form the {\it Riemann sum}
\begin{equation}\label{meas8}
S_{X_j}:=\sum_{a+p^j\Zp\subseteq X}f(x_{a,j})\mu(a+p^j\Zp)\,.
\end{equation}
                                                   
Recall that, as $X$ is compact, every continuous function on $X$ is automatically uniformly 
continuous. Using this, one can directly prove the next basic result.
\begin{theorem}\label{meas9}
As $j\to\infty$, the Riemann sums converge to an element $\int_X f(x)\, d\mu(x) \in\Qp$ which does not depend upon the choice of $X_j$.\end{theorem}
\begin{examples}\label{meas10}
a. Let $\alpha\in X$ and let $\delta_\alpha$ be the Dirac measure at $\alpha$. Then
$\int_Xf(x)\, d\delta_\alpha(x)=f(\alpha)\,.$\\
b. Let $U\subseteq X$ be a compact open with characteristic function $1_U$ (with values $1$ on
$U$ and $0$ elsewhere). Note that $1_U$ is continuous on all of $X$ and that 
$\int_X 1_U\, d\mu(x)=\mu(U)$.
\end{examples}

For now we restrict $X=\Zp$. Let $f\colon \Zp\to \Qp$ be a continuous function. As mentioned above,
the classical Theorem of Mahler tells us that $f(x)$ may be expressed as 
\begin{equation}\label{meas11}
f(x)=\sum_{i=0}^\infty a_i \binom{x}{i}\,,
\end{equation}
where $\{a_i\}\subset \Qp$ and $a_i\to 0$ as $i\to \infty$. Let  $\mu$ be a measure on
$\Zp$ and let $B_\mu$ be an upper bound for $\vert\mu(U)\vert$ where $U$ runs over the 
compact-open subsets of $X$. One that easily checks that 
\begin{equation}\label{meas12}
\left\vert \int_{\Zp} f(x)\, d\mu(x)\right\vert \leq B_\mu\Vert f\Vert\,.
\end{equation}
Now put for all $i\geq 0$
\begin{equation}\label{meas13}
v_{\mu,i}:=\int_{\Zp} \binom{x}{i}\, d\mu(x)\,,
\end{equation}
and notice that $\{v_{\mu,i}\}$ is bounded since one knows, by Mahler, that $\{\binom{x}{i}\}$ is an orthonormal
basis for the Banach space of continuous functions.
We therefore deduce that 
\begin{equation}\label{meas14}
\int_{\Zp} f(x)\, d\mu(x)=\sum_{i=0}^\infty a_i v_{\mu,i}\,, 
\end{equation}
which converges as $a_i\to 0$ as $i\to \infty$.
Conversely, let $V:=\{v_i\}\in \Qp$ be a bounded sequence. Set
$\int_{\Zp} f(x)\, d\mu_V(x):=\sum_i a_iv_i$. As the characteristic functions of compact open subsets are
continuous on all of $\Zp$ and the locally constant functions are dense in ${\mathcal C}(\Zp,\Qp)$,  we obtain a measure by setting $\mu_V(U)=\int 1_U\, d\mu_V(u)$. We thus obtain the next result.

\begin{prop}\label{meas15}
The mapping $\mu\mapsto \{\int \binom{x}{i}\, d\mu(x)\}$ is a vector space isomorphism between the
space of $\Qp$-valued measures and the vector space of bounded sequences $\{v_i\}_{i=0}^\infty$.
\end{prop}
Note that as the locally constant functions are dense in the space of continuous functions, if $\mu$
is a $\Qp$-valued measure and $f\colon X\to \Qp$ is continuous then $f(x)d\mu(x)$ is also
a $\Qp$-valued measure.
\subsection{Convolutions}\label{meas16}
We continue with the $p$-adic set-up of the previous two subsections and we now restrict
to having $X=\Zp$. Let $\mu$ and $\nu$ be two $\Qp$-valued measures on $X$. 
\begin{defn}\label{meas17}
We define the convolution $\mu\ast \nu$ of $\mu$ and $\nu$ by
\begin{equation}\label{meas18}
\int_{\Zp} f(x)\, d\mu\ast \nu(x):=\int_{\Zp}\int_{\Zp} f(x+y)\, d\mu(x)d\mu(y)\,,
\end{equation}
for continuous functions $\Qp$-valued functions on $\Zp$.\end{defn}
\noindent
One checks readily that the convolution of two measures is also a measure and that with this
definition the vector space $\mathcal M_{\Zp,\Qp}$ of $\Qp$-valued measures on
$\Zp$ forms a $\Qp$-algebra. We want to identify this algebra. The first thing to notice is that $\mathcal M_{\Zp,\Qp}=\mathcal M_{\Zp,\Zp}\otimes \Qp$ where the latter is the space of $\Zp$-valued
measures on $\Zp$. 

Let $T$ be an indeterminate and $\Zp[[T]]$ be the formal power series algebra in $T$.
\begin{defn}\label{meas19}
Let $\mu\in \mathcal M_{\Zp,\Zp}$. We define the {\it Mahler transform} $T_M(\mu)\in \Zp[[T]]$ 
of $\mu$ by
\begin{equation}\label{meas20}
T_M(\mu):=\int_{\Zp} (1+T)^x\, d\mu(x):=\sum_{i=0}^\infty \left(\int_{\Zp}\binom{x}{i}\, d\mu(x)\right)T^i\,.
\end{equation}\end{defn}

\begin{prop}\label{meas21}
The Mahler transform $T_M$ is an algebra isomorphism between $\mathcal M_{\Zp,\Zp}$ and
$\Zp[[T]]$.
\end{prop}
\begin{proof}
This follows from the definition of the convolution product and the addition formula for 
binomial coefficients. Indeed, the addition formula precisely corresponds, via convolution,
to the product of the associated power series.
\end{proof}
Let $t_M\colon \Zp[[T]]\to \mathcal M_{\Zp,\Zp}$ be the inverse of $T_M$. 
\begin{prop}\label{meas22}
Let $g(T)\in \Zp[[T]]$ and put $\mu_g=t_M(g)$. We then have for $k\geq 0$
\begin{equation}\label{meas23}
T_M(x^k\, d\mu_g)=D^k g(T)\,,\end{equation}
where $D$ is the continuous operator $(1+T)\frac{d}{dT}$ on $\Zp[[T]]$.\end{prop}
\begin{proof}
Note that $D^k(1+T)^x=x^k(1+T)^x$. Now by definition we have 
\begin{equation}\label{meas24}
\int_{\Zp} (1+T)^x\, d\mu_g(x)=g(T)\,.\end{equation}
Now applying $D^k$ gives
\begin{equation}\label{meas25}
\int_{\Zp} D^k(1+T)^x\, d\mu_g(x)=D^kg(T)\, ,\end{equation}
or
\begin{equation}\label{meas26}
\int_{\Zp}(1+T)^x\, x^kd\mu(x)=D^kg(T)\,.
\end{equation}
and the result follows.
\end{proof}

Proposition \ref{meas22} is obviously the analog of well-known results in the classical Fourier transform.
\subsection{The $\mathfrak v$-adic theory}\label{meas27}
We now return to the case where $A=\Fq[\theta]$ and $\mathfrak v$ is a prime of $A$ with
completion $A_\mathfrak v$. Let $X$ now be a compact open subset of $A_\mathfrak v$.
The basic definitions of $\mathfrak v$-adic measures etc., on $X$ follow directly  
{\it mutatis mutandis} from the theory of Subsections \ref{meas1} and \ref{meas7}. We therefore
focus here on the $\mathfrak v$-adic version of Subsection \ref{meas16}.

Therefore let $X=A_\mathfrak v$ and let $f\colon A_\mathfrak v\to k_\mathfrak v$ be a continuous
function. By the theorem of Wagner, Corollary \ref{vadic8}, we have an expression
\begin{equation}\label{meas28}
f(x)=\sum_{i=0}^\infty a_j G_j (x)\, ,
\end{equation}
where $\{a_j\}\subset k_\mathfrak v$ and $a_j\to 0$ as $j \to \infty$. If $\mu$ is a 
$\mathfrak v$-adic measure then $\mu$ is determined by the bounded sequence $\left\{\displaystyle
\int_{A_\mathfrak v} G_j(x)\, d\mu(x)\right\}_{j=0}^\infty$. Conversely any such bounded sequence
determines a $\mathfrak v$-adic measure. Under convolution of measures, such sequences form
a $k_\mathfrak v$-algebra denoted $\mathcal M_{A_\mathfrak v, k_\mathfrak v}$ and we would like to identify this algebra. As in the $p$-adic case, we immediately reduce to identifying the
algebra $\mathcal M_{A_\mathfrak v,A_\mathfrak v}$ of $A_\mathfrak v$-valued measures
on $A_\mathfrak v$. 

Let $\left\{\frac{u^i}{i!}\right\}_{i=0}^\infty$ be the divided power elements. These are 
formal symbols with the usual multiplication rule 
\begin{equation}\label{meas28.1}
\frac{u^i}{i!}\cdot \frac{u^j}{j!}:=\binom{i+j}{i}\frac{u^{i+j}}{(i+j)!}=\binom{i+j}{j}\frac{u^{i+j}}{(i+j)!}\,,
\end{equation}
so that the multiplication is well defined in any characteristic.

\begin{defn}\label{meas29} 
We define  $A_\mathfrak v \{\{ u\}\}$ to be the algebra of formal divided power series
\begin{equation}
\sum_{i=0}^\infty a_i \frac{u^i}{i!}\,\end{equation}
with the obvious addition and with multiplication given above.
\end{defn}

\begin{defn}\label{meas30}
Let $\mu\in \mathcal M_{A_\mathfrak v,A_\mathfrak v}$. We set
\begin{equation}\label{meas31} 
T_W(\mu):=\sum_{j=0}^\infty \left(\int_{A_\mathfrak v} G_j(x)\, d\mu(x) \right)  \frac{u^j}{j!}\in  A_\mathfrak v\{\{ u\}\}\,.\end{equation}
We call $T_W(\mu)$ the {\it Wagner transform} of $\mu$.\end{defn}

The addition formula for $G_j(x)$ given in Proposition \ref{def15} immediately gives our
next result which is the analog of Proposition \ref{meas21}.

\begin{prop}\label{meas32} 
The Wagner transform $T_W\colon \mathcal M_{A_\mathfrak
v, A_\mathfrak v} \to A_\mathfrak v \{\{u \}\}$ is an algebra isomorphism.\end{prop}

\begin{rems}\label{meas33}
a. The present author had used the addition formula to calculate the convolution of measures 
which Greg Anderson realized was isomorphic to the algebra of formal divided  series.\\
b. There is another equivalent representation of the above algebra. Let $R= A_v[z]$ be the polynomial ring in $z$. Let
$\partial_j=\partial_{z,j}\colon R\to R$ be the {\it $j$-divided derivative} operator defined, as usual, by
$\partial_j z^i :=\binom{i}{j}z^{i-j}$. Note that $\partial_i\partial_j=\binom{i+j}{i}\partial_{i+j}$. Let $A_\mathfrak v\{\{\partial\}\}$ be the algebra of formal divided derivatives formed
in the obvious fashion (which obviously acts on $R$). One then readily sees that $A_\mathfrak v\{\{\partial\}\}$ is
isomorphic to $A_\mathfrak v\{\{u\}\}$.

\end{rems}

\begin{defn}\label{meas34}
Let $f\in {\mathcal C}(A_\mathfrak v, k_\mathfrak v)$ and let $\mu \in \mathcal M_{A_\mathfrak v,k_n}$. 
we set
\begin{equation}\label{meas35}
\mu\ast f(x):=\int_{A_\mathfrak v} f(x+y)\, d\mu(y)\in {\mathcal C}(A_{\mathfrak v}, k_{\mathfrak v})\,.
\end{equation}\end{defn}

As in \cite{Go05}, we can now present an analog of Proposition \ref{meas22} in the
finite characteristic theory. Let $z$ be a variable,
as above,
and let $k_v\langle \langle z \rangle \rangle$ be the {\it Tate algebra} of $k_v$ power series $\sum_i a_iz^i$ where
 $a_i\in k_v\to 0 $  as $i\to \infty$. Such power series converge on those $z \in k_v$ with $\mathfrak v$-adic
 absolute value $\vert z \vert_\mathfrak v\leq 1$ (i.e., $z \in A_\mathfrak v$).  The algebra
 $A_\mathfrak v\{\{\partial\}\}$ acts on $k_v\langle\langle  z\rangle \rangle$ in the
 natural fashion.
 
 \begin{defn}\label{meas36}
 Let $z\in A_\mathfrak v$ and let $\mu_z$ be the measure
 whose Wagner transform is $\sum_i z^i \frac{u^i}{i!}$. If $f\in {\mathcal C}(A_\mathfrak v, k_v)$
 we set
 \begin{equation}\label{meas37}
 \hat f(z):=\int_{A_\mathfrak v}f(t)\, d\mu_z(t)\in k_\mathfrak v\langle \langle z \rangle \rangle\, .\end{equation}
 \end{defn}
 \noindent
 Thus if $f(x)=\sum a_iG_i(x)$ then $\hat f (z)=\sum a_iz^i$. Moreover,
 the map $f\mapsto \hat f$ is a Banach space isomorphism between
  ${\mathcal C}(A_\mathfrak v, k_\mathfrak v)$ and $k_\mathfrak v\langle \langle z \rangle \rangle$.
 
 Finally, by a small abuse of notation, let us also denote by ``$\frac{u^i}{i!}$'' the measure
 whose Wagner transformation is $\frac{u^i}{i!}$. The following, which is an
 analog of Proposition \ref{meas22}, also follows directly from
 the addition formula Proposition \ref{def15}.
 
 \begin{prop}\label{meas38}
Let $f$ be a continuous $k_\mathfrak v$-valued function on $A_\mathfrak v$. With the above
definitions, we have
\begin{equation}\label{meas39}
\widehat {\frac{u^i}{i!}\ast f}=\partial_i\hat f\,.
\end{equation}
\end{prop}
\noindent
In other words, convolution transforms into differentiation.

\begin{rems}\label{meas40}
The very important work of B.\ Angl\`es, F.\ Pellarin and others (see \cite{AP14} for instance) uses Tate algebras in
many variables $t_1,\dots, t_m$ to study $L$-series, where $m$ can be an arbitrary
positive integer and this suggests that one might need a multi-variable transform.
On the other hand, while
we have not stressed it here, it is well-known that the convolution algebra of $p$-adic
measures on $\Zp$ is isomorphic to the completed group algebra $\Zp[[\Zp]]$. Similarly
the algebra of $\mathfrak v$-adic measures on $A_\mathfrak v$ is isomorphic to the 
completed group ring $A_\mathfrak v[[A_\mathfrak v]]$, etc. 

However, as a topological group $A_\mathfrak v$ is isomorphic to the infinite product
of $\Fq$ with itself equipped with the product topology. As such, $A_\mathfrak v$ is  therefore isomorphic
to $A_\mathfrak v^m$ as topological groups, and this isomorphism clearly extends to the
complete group rings. This {\em suggests}, therefore, that one should be able to construct
versions of Propositions \ref{meas32} and \ref{meas38} for many variables $\{z_1,\dots, z_m\}$ and probably in many
different ways.
\end{rems}
 
 \section{$L$-series}\label{series}
 In this section we review the definitions of $L$-series for $A=\Fq[\theta]$,  and $k=\Fq(\theta)$.
 Let $\infty$ be the infinite prime of $k$ with normalized absolute value $\vert ?\vert_\infty$ and
 completion $k_\infty$. We set $\C_\infty$ to be the completion of a fixed algebraic closure
 $\bar k_\infty$ equipped with the canonical extension of $\vert ?\vert_\infty$. 
 
 \subsection{Exponentiation}\label{series1}
 In order to define exponentiation of elements, we must begin with a notion of ``positivity" in 
 $k_\infty$. Let $\pi\in k_\infty$ with $\vert \pi \vert_\infty=1$. Let $x\in k_\infty^\ast$ be expressed
 as $\sum_{i=e}^\infty c_i\pi^i$ where $\{c_e\}\subseteq \Fq$, $e$ is an integer (positive, negative or $0$), and $c_e\neq 0$.
 
 \begin{defn}\label{series2}
 We set $ {\sgn}(x):=c_e$ and
say that $x$ is {\it positive} (or monic) if and only $c_e=1$ We call $-e$ the {\it degree} of $x$
and denote it $\deg x$. 
 \end{defn}
 
 \begin{rems}\label{series3}
1.  It is easy to see that ${\sgn}$ is a homomorphism from $k_\infty^\ast\to \Fq^\ast$
which is the identity on $\Fq^\ast$ and $1$ on principal units (elements of the form $1+c\pi+\cdots$,
with $c\in \Fq$). This mapping $\sgn$ is called a {\it sign function}.\\
2. Conversely given a homomorphism $\sgn$ as in Part 1, we deduce an associated
notion of positivity.\\
3. Notice that the positive elements $k_{\infty,+}$ form a subgroup of $k_\infty^\ast$, but are obviously not closed
under addition.\\
4. The notion of degree given above agrees with the usual notion of degree in $A$.
\end{rems}

\noindent
{\bf Convention:} For simplicity we now fix the unique sign function $\sgn$ with $\sgn \theta=1$.

\begin{defn}\label{series4}
We let $\pi\in k_\infty$ be a fixed positive uniformizing element at $\infty$.
\end{defn}

\noindent
As an example, one could choose $\pi:=1/\theta$; in general, of course, $\vert \pi\vert_\infty=\vert 1/\theta\vert_\infty$.

Let $\alpha\in k_{\infty,+}$. With the above definitions we then deduce a canonical decomposition
\begin{equation}\label{series5}
\alpha=\pi^{-\deg \alpha}\langle \alpha \rangle\,,
\end{equation}
where $\langle \alpha \rangle \equiv 1\pmod{\pi}$. Of course $\langle \alpha\rangle$ depends on
$\pi$ but we shall not explicitly mention this for simplicity once $\pi$ is chosen.

Now let $s=(x, y)$ where $x\in \C_\infty^\ast$,   and $y\in \Zp$.
\begin{defn}\label{series6}
We set 
\begin{equation}\label{series7}
\alpha^s:=x^{\deg \alpha}\langle \alpha \rangle^y \,,
\end{equation}
where $\langle \alpha \rangle^y$ is defined, as usual, via the binomial theorem.
\end{defn}
The elements $s$ form a group under the obvious definitions whose operation will be written
additively. One then sees that, as usual, $(\alpha \beta)^s=\alpha^s\beta^s$ and 
$\alpha^{s_0+s_1}=\alpha^{s_0}\alpha^{s_1}$. 

\begin{defn}\label{series8}
 We set $\mathbb S_\infty:=\C_\infty^\ast \times \Zp$ with the obvious structure as a
 topological additive group.
 \end{defn}
 \begin{rems}\label{series9}
 Let $i\in \Z$  and set $s_i:=(1/\pi^i,i)\in {\mathbb S}_\infty$. Let  $\alpha \in k_{\infty,+}$. Then by
 definition one sees that $\alpha^{s_i}=\alpha^i$ where $\alpha^i$ has its usual meaning.
 \end{rems}

\begin{defn}\label{series10}
We call a formal sum $L(s):=\displaystyle \sum_{a\in A_+} c_a a^{-s}$, where $\{c_a\}\subset \C_\infty$
and $s\in {\mathbb S}_\infty$
a {\it Dirichlet series}. \end{defn}

The convergence properties of such a Dirichlet series $L(s)$ are determined by the
 $x\in \C_\infty^\ast$
coordinate. In practice one ends up with a family of entire power series in $x^{-1}$ which is continuous
on all of $\hat{\mathbb S}_\infty$; see, for example, Section 8 of \cite{go96} or \cite{Go05} for more details.

\begin{example}\label{series11}
We set $\zeta_A (s):=\sum_{a\in A_+}a^{-s}$ and call it the {\it zeta function of $A$}. It is a 
standard exercise to express $\zeta_A(s)$ as an Euler product over all the monic irreducible 
elements in $A$.
\end{example}

Let $L(s)=\sum_{a\in A_+} c_a a^{-s}$ be a Dirichlet series. By definition
we have 
\begin{equation}\label{series12}
L(s)=\sum_{d=0}^\infty x^{-d}\left( \sum_{\deg a =d} c_a \langle a \rangle^{-y}\right)\,\end{equation}
where $s=(x,y)\in {\mathbb S}_\infty$. The convergence properties of $L(s)$ are then
determined by the convergence properties of the above family of power series in $x^{-1}$.
In all known cases, these power series are shown to be entire.

\begin{rems}\label{series13}
a. Our theory starts with a choice of a positive uniformizer $\pi_0$. Given another uniformizer
$\pi_1$, clearly $u=\pi_1/\pi_0$ is a $1$-unit at $\infty$. Conversely, given a $1$-unit $u$ the 
element $u\pi_0$ is another positive uniformizer.  Moreover $\pi_1^{\deg \alpha}\alpha=
(\pi_1/\pi_0)^{\deg \alpha} \pi_0^{\deg \alpha} \alpha$ so that the ``gauge" $\pi_1/\pi_0$ tells
us how to pass between the various choices of positive uniformizer \cite{Go11}.\\
b. The general theory of Drinfeld modules is defined for all rings $A$ given as the algebra of
regular functions outside a fixed point $\infty$ on a smooth projective geometrically connected
curve over $\Fq$. In this case, just exponentiating elements is not sufficient but in fact
our definitions may be readily extended to exponentiating general nonzero ideals of $A$, see
for instance \cite{Go11}.
\end{rems}
\subsection{A generalization due to B.\ Angl\`es}\label{bruno}
We briefly present here a remarkable generalization of the zeta function of Example \ref{series11}
recently communicated to us by Angl\`es \cite{An16} (see also \cite{ADR16}) which we follow rather closely with his permission.

Therefore let $F$ be a field containing $\Fq$ which is complete for a valuation $v\colon F\to \R\cup \{+\infty\}$.

\begin{defn}\label{bruno1} Let $n$ be a positive integer. We set
\begin{equation}\label{bruno2}
{\mathbb S}_{F,n}:=F^\ast\times \Zp^n\,,
\end{equation}
with the obvious structure as an additive topological group.\end{defn}
\noindent
Clearly ${\mathbb S}_{\C_\infty,1}=\mathbb S_\infty$ where the latter is defined in Definition \ref{series8}. 

Now let $t$ be a variable and put $\mathcal A:=\Fq[t]$ with $\mathcal A_+$ being the subset
of monics and $\mathcal A_{+,d}$ being those of degree $d$. Mimicking our earlier constructions for
$A=\Fq[\theta]$ (and $\pi=1/\theta$) for $a(t)\in \mathcal A_+$
\begin{equation}\label{bruno3}
\langle a\rangle =t^{-\deg a}a\,.
\end{equation}
\begin{defn}\label{bruno4} 
a. Let $\{z_1,\ldots,z_n\}\subset F$ with $v(z_i)<0$ for all $i$ (so that $\langle a(z_i)\rangle=\langle a \rangle\vert_{t=z_i}$
converges to an element of $F$).  Let $s=(x,y_1,\dots,y_n)\in
\mathbb S_{F,n}$. We set
\begin{equation}\label{bruno5}
a^s(z_1,\dots,z_n)=x^{\deg a}\langle a(z_1)\rangle^{y_1}\cdots \langle a(z_n)\rangle^{y_n}\,.
\end{equation}
\noindent
b. We set
\begin{equation}\label{bruno6}
\zeta(s)(z_1,\ldots,z_n):=\sum_{a\in \mathcal A_+} a^{-s}(z_1,\ldots, z_n)\,.
\end{equation}

\end{defn}
\begin{example}\label{bruno7}
With $F=\C_\infty$, $n=1$, and $z_1=\theta$, we recover $\zeta_A(s)$. 
\end{example}

Notice that by definition $\zeta(s)(z_1,\ldots,z_n)$ converges for all $s$ with $v(x)<0$ (as before).
Note also that we have
\begin{equation}\label{bruno8}
\zeta(s)(z_1,\ldots,z_n)=\sum_{d\geq 0} x^{-d}\left(\sum_{a\in \mathcal A_{+,d} }\langle a(z_1)\rangle^{-y_1}
\cdots \langle a(z_n)\rangle^{-y_n}\right)\,.
\end{equation}

\begin{prop}\label{bruno9}
The function $\zeta(s)(z_1,\cdots,z_n)$ converges (in $F$) for all $s\in \mathbb S_{F,n}$.
\end{prop}
\begin{proof}
Let $\{t_1,\ldots, t_j\}$ be $j\geq 1$ indeterminates. Standard estimates (Chapter 8 
\cite{go96}) show that for
$k<d(q-1)$ we have $\sum_{a\in A_{+,d}}\prod_i a(t_i)=0$ (which is sometimes referred to as
{\it Simon's Lemma}). Notice, of course, that we are free to substitute $t_i^{q^{e_i}}$ for $t_i$. 
Now let $m=\sum_ic_iq^i$ be a positive integer written $q$-adically (so that $0\leq c_i<q$ all $i$) with
$\ell_q(m):=\sum_ic_i$ the sum of its $q$-adic digits. We then conclude that  if $m_1,\ldots, m_n$ are positive integers with
$\sum_{i=1}^n \ell_q(m_i)<d(q-1)$ then
\begin{equation}\label{bruno10}
\sum_{a\in \mathcal A_{+,d}}a(z_1)^{m_1}\cdots a(z_n)^{m_n}=0\,.
\end{equation}
Now let $\{y_i\}_{i=0}^n\subset \Zp$ and express $-y_i$ $q$-adically as $-y_i=\sum_j u_{i,j}q^j$

Pick a positive integer $l$ and write $-y_i=m_i+r_i$ where $m_i=\sum_{j=0}^lu_{i,j}q^j$. 
Notice that 
\begin{equation}\label{bruno11}
\langle a(z_i)\rangle^{r_i}=cz_i^{-q^{l+1}}+\{\text {higher~terms}\}\,,\end{equation}
where $c\in \Fq$.
Note also that $\ell_q(m_i)\leq (l+1)(q-1)$ so that $\sum \ell_q(m_i) \leq n(l+1)(q-1)$. 
Thus, using the definition of $\langle a(t)\rangle$ and Equation \ref{bruno10}, we conclude that 
$$\sum_{a\in A_{+,d}} \langle a(z_1) \rangle^{m_1}\cdots \langle a(z_n)\rangle^{m_n}=0$$ for
$d$ such that $n(l+1)(q-1)<d(q-1)$ or $d>(l+1)n$. Moreover in this case we conclude
that 
\begin{equation}\label{bruno12}
v\left(\sum_{a\in \mathcal A_{+,d}}\langle a(z_1)\rangle^{-y_1}\cdots\langle a(z_n)\rangle^{-y_n}\right) 
\geq  q^{l+1}\inf\{-v(z_i)\}\,.
\end{equation}
Let $[x]$ be the greatest integer of a real number $x$ and choose $l$ so that $l=[d/n]-2$ 
(which ensures $d>(l+1)n$). We then conclude
\begin{equation}\label{bruno13}
v\left(\sum_{a\in \mathcal A_{+,d}}\langle a(z_1)\rangle^{-y_1}\cdots\langle a(z_n)\rangle^{-y_n}\right) 
\geq q^{[d/n]-1}\inf\{ -v(z_i)\}\,,
\end{equation}
which is much stronger than what is needed to conclude everywhere convergence.
\end{proof}

As a corollary, one immediately obtains the analytic continuation of $\zeta_A(s)$. But there are
other, higher order, corollaries of Proposition \ref{bruno9} that we now turn to.
\begin{defn}\label{bruno14}
Let $\mathbb T_n(\C_\infty)$ be the {\it Tate algebra} of power series $f(t_1,\ldots, t_n)$ in $\{t_1,\ldots, t_n\}$ 
converging in the unit polydisc in $\C_\infty^n$.\end{defn}

 In other words, let $J=\{j_1,\ldots, j_n\}$ be
a multi-index; set $t^J:=\prod_i t_i^{j_1}$ and $\Vert J\Vert:=\sum j_i$. Then $\mathbb T_n(\C_\infty)$
consists of those power series $f(t)=\sum_J c_J t^J$ with $\vert c_J\vert_\infty \to 0$ as
$\Vert J\Vert\to \infty$. We set $\Vert f\Vert=\max_J \{\vert c_J\vert_\infty\}$ which is the
{\it Gauss norm}. 

\begin{defn}\label{bruno15}
Let $a\in \mathcal A_{+,d}$. We set 
\begin{equation}\label{bruno16}
\langle a(t_i+\theta) \rangle_\circ:=\frac{a(t_i+\theta)}{\theta^{\deg a}}\,.
\end{equation}\end{defn}

\noindent

Notice that $\langle a(t_i+\theta)\rangle_\circ$ is a deformation of the definition of $\langle a(\theta)\rangle$ given in Equation \ref{series5} with $\pi=1/\theta$. Indeed, 
\begin{equation}\label{bruno17}
\langle a(t_i+\theta)\rangle_\circ\vert_{t_i=0}=\langle a(\theta)\rangle\,.\end{equation}
Notice also that $\langle a(t_i+\theta)\rangle_\circ\in 1+\frac{1}{\theta} \Fq[t_i][\frac{1}{\theta}]$. 

As in Subsection \ref{det}, we let $\partial_j=\partial_{\theta,j}$ be the $j$-th divided derivative
in $\theta$ of $f(\theta)$ and set $f^{(m)}=\partial_mf(\theta)$. Therefore  if
$a\in \mathcal A_+$, then $a(t_i+\theta)=\sum_m a^{(m)}t_i^m$ (as usual from calculus).

Let $s=(x,y_1,\ldots, y_n)\in \mathbb S_{\C_\infty,n}$ and let $a\in \mathcal A_+$.  
\begin{defn}\label{bruno18}
1. We set 
\begin{equation}\label{bruno19} 
a^s_\circ:=x^{\deg a}\langle a(t_1+\theta)\rangle^{y_1}_\circ\cdots \langle a(t_n+\theta)\rangle_\circ^{y_n}\,.\end{equation} Note that $a^s_\circ\in \mathbb T_n(\C_\infty)$.

\noindent
2. We set 
\begin{equation}\label{bruno20}
L^\circ_{n}(s):=\sum_{a\in \mathcal A_+}a^{-s}_\circ\,.
\end{equation}\end{defn}
\noindent
Unwinding the definition gives immediately
\begin{equation}\label{bruno21}
L^\circ_n(s)=\sum_{d=0}^\infty x^{-d}\left(\sum_{a\in \mathcal A_{+,d}} \langle a(t_i+\theta)\rangle_\circ^{-y_1}\cdots \langle a(t_n+\theta)\rangle_\circ^{-y_n}\right)\,.\end{equation}

The Gauss norm gives an absolute value on the quotient field of $\mathbb T_n(\C_\infty)$ and
we now let $F$ be the completion under this norm. 

\begin{lemma}\label{bruno22}
As functions on $\mathbb S_{F,n}$ we have
\begin{equation}\label{bruno23}
L^\circ_n(s)=\zeta\left(x\cdot\prod_{i=1}^n(1+t_i/\theta)^{-y_i}, y_1,\ldots, y_n\right)(t_1+\theta,
\cdots ,t_n+\theta)\,.\end{equation}
\end{lemma}
\begin{proof}
Note that $\Vert t_i+\theta\Vert =\vert \theta\vert_\infty >1$ so that under the additive
valuation $v$ associated to the Gauss norm, we have $v(t_i+\theta)<0$. The rest of the proof
involves unraveling the definitions.\end{proof}

We then have the following extremely important corollary.

\begin{cor}\label{bruno24}
Let $n$ be a positive integer and let $\{m_i\}_{i=1}^n$ and $\{k_j\}_{j=0}^n $ be two collections of
$n$ nonnegative integers. Then the following function on $\mathbb S_\infty$ (defined via
the uniformizer $\pi:=1/\theta$)
\begin{equation}\label{bruno25}
\sum_{d=0}^\infty \sum_{a\in A_{+,d}} a^{-s}(a^{(m_1)})^{k_1}\cdots(a^{(m_n)})^{k_n}
\end{equation}
is entire (i.e., is a continuous family of entire power series in $x^{-1}$).
\end{cor}
\begin{proof}
Notice that
\begin{equation}\label{bruno26}
L_{n+1}^\circ(\theta^{-n}x,-1,\ldots,-1,y)\vert_{t_{n+1}=0}=\sum_{i_1,\dots,i_n}\left(
\sum_{d=0}^\infty x^{-d} \sum_{a\in A_{+,d}} \langle a \rangle^{-y}a^{(i_i)}\cdots a^{(i_n)}\right)t_1^{i_1}\cdots t_n^{i_n}\,.
\end{equation}
The result now follows from Lemma \ref{bruno22} and Proposition \ref{bruno9}.\end{proof}

\begin{rems}\label{bruno27}
1. Note also that $$L_{n+1}^\circ(\theta^{-n}x,-1,\ldots,-1,y)\vert_{t_{n+1}=0}=\sum_d \sum_{a\in A_{+,d}}
a(t_1+\theta)\cdots a(t_n+\theta) a^{-s}$$ for $s=(x,y)\in \mathbb S_\infty$. 
This is a similar to certain $L$-series
previously discussed in \cite{Pe12} and \cite{AP14}.\\
2. The fundamental estimates of Chapter 8 \cite{go96}, show that the sum in Equation \ref{bruno23} is 
``essentially algebraic." Indeed, let $\mathfrak v$ be a prime of
$A$. Then variants of the arguments given here establish that the 
series in Equation \ref{bruno25} can also be interpolated to entire $\mathfrak v$-adic
functions on the appropriate $\mathfrak v$-adic analog of $\mathbb S_\infty$.\end{rems}

\section{Continuous functions on $\Zp$ into finite characteristic complete algebras}\label{zptofinite}

Let $L$ be a finite extension of $\Qp$ with ring of integers $\mathcal O$ equipped with the canonical 
topology. Let $f\colon\Zp \to \mathcal O$ be a continuous function. As mentioned $f$ has a canonical expansion
$f(y)=\sum_{j=0}^\infty a_j \binom{y}{j}$ where $\{a_j\}\subset \mathcal O$ and $a_j\to 0$ as $j\to \infty$; conversely any such sequence $\{a_j\}$ uniquely defines a continuous function. As $\Zp$ is
compact, one readily deduces Mahler expansions for all continuous $L$-valued functions on $\Zp$.

 In Section 8.4 of \cite{go96} it is shown that the obvious variant of Mahler's Theorem also
holds true if $\mathcal O$ is a complete algebra over $\Z/(p)$ where now $\binom{y}{j}$ is reduced
modulo $p$ to obtain a function from $\Zp$ to $\mathcal O$.

\subsection{Dirichlet series on $\Zp$}\label{dirichseries}

Let $s=(x,y)\in \mathbb S_\infty$  and let $L(s)=\sum_{a\in A_+} c_aa^{-s}$ be an $L$-series as in the previous section.  Write $\langle a \rangle =1+v_a$ where $v_a\equiv 0\pmod{1/\theta}$ so that
$(1+v_a)^{-1}=1+w_a$ where $w_a=-v_a+v_a^2+\cdots$. 
Therefore $\langle a\rangle ^{-y}=(1+w_a)^y $ and its Mahler expansion is then immediate from
the binomial expansion. In this fashion the Mahler coefficients of $L(x,y)$ can be computed in terms
of $\{c_a\}$ and powers of $x$. 

\begin{defn}\label{dirichlet1}
Let $L$ be a finite extension of $k_\infty=\Fq((1/\theta))$ with ring of integers $\mathcal O:=\mathcal O_L$.
Let $\mathcal D:=\mathcal D(\Zp,\mathcal O)$ be the closure in the Banach space $\mathcal C(\Zp,L)$ (of continuous functions
from $\Zp$ to $L$) of $\mathcal O$-linear combinations of functions of the form $y\mapsto u^y$ for
$u\in \mathcal O_L$ a principal unit. Following  W.\ Sinnott \cite{Si08}, we call the elements of $\mathcal D$ 
{\it Dirichlet series on $\Zp$ with values in $\mathcal O_L$.}\end{defn}

The following three fundamental results are then established in \cite{Si08} and we refer the reader there for 
details. To begin, let $U_1\in \mathcal O$ be the group of principal units $u$; i.e., $u\equiv 1
\pmod M$ where $M\subset \mathcal O$ is the maximal ideal.. As in Subsection \ref{meas27},
one has the now obvious notion of $L$ valued measures on $U_1$ equipped with the standard 
notion of convolution (using the {\it multiplicative} group action on $U_1$). Let $\mathcal M(U_1,L)$ 
(respectively, \ $\mathcal M(U_1,\mathcal O)$) the space of measures with coefficients in $L$ (respectively, $\mathcal O$).  

\begin{defn}\label{dirichlet2} Let $\mu\in \mathcal M (U_1,\mathcal O)$. We define its 
{\it $\Gamma$-transform} $\Gamma_\mu\colon \Zp\to \mathcal O$ by
\begin{equation}\label{dirichlet3}
\Gamma_\mu(y):=\int_{U_1} u^y\, d\mu(u)\,.\end{equation}
\end{defn}

\begin{theorem}\label{dirichlet4}
The Gamma-transform is an isomorphism of topological $\mathcal O$-modules between
$\mathcal M(U_1,\mathcal O)$ and $\mathcal D(\Zp,\mathcal O)$.\end{theorem}

\begin{theorem}\label{dirichlet5}
The $L$-span of $\mathcal D$ is dense in $\mathcal C(\Zp, L)$.
\end{theorem}

The next, and last result of \cite{Si08} we mention, is remarkable because it shows that, in some sense,
the elements of $\mathcal D$ exhibit behavior similar to that of analytic functions.

\begin{theorem}\label{dirichlet6}
Let $f,g\in \mathcal D$. Suppose that $f$ and $g$ coincide in a neighborhood of a point in $\Zp$.
Then $f=g$.\end{theorem}

\section{The group $S_{(q)}$}\label{sq}
The analogy with the complex theory leads one to look for functional equations (i.e., symmetries) of
our $L$-series. Looking at the special values of the special case $\zeta_A(s)=\sum_{a\in A_+}a^{-s}$, 
it is apparent that if there are symmetries, they will not be of the classical form $s\mapsto 1-s$.

In \cite{Th95}, D.\ Thakur presented some calculations related to trivial zeroes of zeta functions for some
non-polynomial base rings. Based on these calculations we were led to the following construction.

\begin{defn}\label{sq1}
 Let $\rho$ be a permutation of $\{0,1,\ldots,\}$ and let $y\in \Zp$ be written $q$-adically as $\sum_{j=0}^\infty c_jq^j$, $0\leq c_j<q$ all $j$. We set
\begin{equation}\label{sq2}
\rho_\ast y:=\sum_{j=0}^\infty c_jq^{\rho j}\,.
\end{equation}\end{defn}
Clearly $y\mapsto \rho_\ast y$ is a bijection of $\Zp$; the set of these bijections forms a group
denoted $S_{(q)}$ with the cardinality of continuum. 

In \cite{Go11} the following result is established.
\begin{prop}\label{sq2.1}
1. The map $y\to \rho_\ast y$ is a homeomorphism of $\Zp$.\\
2.  The mapping $y\to \rho_\ast y$ stabilizes both the nonnegative and nonpositive integers.\\
3. Let $y_0,y_1$ be two $p$-adic integers such that $y_0+y_1$ has no carryover of $q$-adic digits. Then $\rho_\ast (y_0+y_1)=\rho_\ast(y_0)+\rho_\ast(y_1)$. \\
4. Let $n$ be a nonnegative integer with sum of $q$-adic digits $\ell_q(n)$. Then $\ell_q(n)=\ell_q(
\rho_\ast n)$.\\
5. Let $n$ be an integer. Then $n\equiv \rho_\ast n\pmod{q-1}$.\\
6. Let $n, j$ be two nonnegative integers. Then $\binom{n}{j}\equiv \binom{\rho_\ast n}{\rho_\ast j}
\pmod{p}$.
\end{prop}

\noindent
Note that the last statement of the proposition follows directly from Lucas' congruence.

\subsubsection{The first action of $S_{(q)}$ on continuous $\mathcal O$-valued functions}\label{sq3}
As the action of $S_{(q)}$ on $\Zp$ consists of homeomorphisms, we obtain a natural action 
on the space of continuous functions $\mathcal C(\Zp,\mathcal O)$.

\begin{defn}\label{sq4}
Let $f\in \mathcal C(\Zp,\mathcal O)$ and $\rho_\ast\in S_{(q)}$. We define $f^{\rho_1}\in
\mathcal C(\Zp,\mathcal O)$ by $f^{\rho_1}(y):=f(\rho_\ast^{-1}y)$.\end{defn}

\noindent
Note that this action is an automorphism of the algebra $\mathcal C(\Zp,\mathcal O)$.

\begin{prop}\label{sq5} As functions from $\Zp$ to $\mathcal O$,
we have $\binom{y}{j}^{\rho_1}=\binom{y}{\rho_\ast j}$.\end{prop}
\begin{proof}
By Part 6 of Proposition \ref{sq2.1} we have, as functions from $\Zp$ to $\mathcal O$,
$\binom{\rho_\ast^{-1}y}{\rho_\ast^{-1}j}=\binom{y}{j}$. Now substitute $\rho_\ast j$ for $j$.
\end{proof}

\begin{cor}\label{sq6}
Let $f\in \mathcal C(\Zp,\mathcal O)$ have Mahler expansion $f(y)=\sum_jc_j \binom{y}{j}$. Then
\begin{equation}\label{sq7}
f^{\rho_1}(y)=\sum_j c_j \binom{y}{\rho_\ast j}\,.
\end{equation}\end{cor}

\subsubsection{An associated congruence}\label{sq8}
We begin here by considering binomial coefficients as taking values in $\Zp$. 
Let $a,b$ be two nonnegative integers. Then the function $y\mapsto \binom{y}{a}\binom{y}{b}$
is also a continuous function from $\Zp$ to itself and therefore has a well-known associated Mahler expansion:
\begin{equation}\label{sq9} 
\binom{y}{a}\binom{y}{b}=\sum_{k=0}^a \binom{a+b-k}{k, a-k,b-k}\binom{y}{a+b-k}\, . \end{equation}

\begin{prop}\label{sq10} 
Let $\rho_\ast \in S_{(q)}$. Let $n,m$ be nonnegative integers and $y\in \Zp$.
Set $$s_1:= \sum_{k=0}^m \binom{m+n-k}{k,m-k,n-k}\binom{y}{\rho_\ast (m+n-k)}\,,$$
and
$$s_2:=\sum_{k=0}^{\rho_\ast m}\binom{\rho_\ast m+\rho_\ast n-k}{k,\rho_\ast m-k,\rho_\ast n-k}
\binom{y}{\rho_\ast m+\rho_\ast n-k}\,.$$
Then $s_1\equiv s_2\pmod{p}$.\end{prop}
\begin{proof}
The first sum, $s_1$, arises from applying $\rho$ to the reduction modulo $p$ of Equation
\ref{sq9}. Next, note that, as functions from $\Zp$ to $\mathcal O$, we have
\begin{equation}\label{sq11} \left
(\binom{y}{m}\binom{y}{n}\right)^{\rho_1}=\binom{y}{m}^{\rho_1}\binom{y}{n}^{\rho_1}=
\binom{y}{\rho_\ast m}\binom{y}{\rho_\ast n}\,. \end{equation}
Now substituting $\rho_\ast n$ for $b$ and $\rho_\ast m$ for $a$ in Equation \ref{sq9}. As both sums represent the same function modulo $p$, we are finished.\end{proof} 

In general, the sums in Proposition \ref{sq10} do not have the same number of elements.
Indeed, the reader can easily construct examples where the difference between $\rho_\ast m$ and $m$ is arbitrarily large.

\subsubsection{The space $\mathcal D$ is not stable under the first action of $S_{(q)}$}\label{sq12}

Recall that in Definition \ref{dirichlet1} we presented Sinnott's space $\mathcal D$ of Dirichlet
series on $\Zp$ with values in $\mathcal O$. By example we show here that the space 
$\mathcal D$ is {\it not} stable under $S_{(q)}$.

\begin{example}\label{sq13} Let $\rho$ be the permutation of $\{0,1,\dots\}$ obtained by exchanging
$0$ and $1$ and fixing all other integers. Let $\mathcal O=\Fq[[1/\theta]]$, $u=1+1/\theta$
and $f(y):=u^y\in \mathcal D$. Let $g(y):=f^{\rho_1}(y)=f(\rho_\ast y)$ (as $\rho^2=1$). Clearly
$f(y)=g(y)$ for $y\equiv 0 \pmod{p^2}$. Thus if $g\in \mathcal D$, we would have $f=g$ by
Theorem \ref{dirichlet6} which is easily seen to be false.
\end{example}

In Subsection \ref{sq24} we will present an action of $S_{(q)}$ which {\em does} preserve
the space $\mathcal D$.

\subsubsection{The first action of $S_{(q)}$ on measures}\label{sq14}
Recall that in Proposition \ref{meas32} we showed that the Wagner transform gave an
isomorphism between the algebra of measures on $A_\mathfrak v$ and the $A_\mathfrak v$
algebra of divided power series. Here we establish a natural action of $S_{(q)}$ on the algebra
of formal divided power series over a ring $R$ of characteristic $p$. 

Let $u$ be an indeterminate. Let $i$ be a nonnegative integer and
$\rho\in S_{(q)}$. 

\begin{defn}\label{sq15}
We set
\begin{equation}\label{sq16}
\rho_\ast u^i:=u^{\rho_\ast i}\,.\end{equation}\end{defn}

Clearly $\rho_\ast$, as just defined, extends to an $R$-module automorphism of $R[[u]]$. Note
that if $i,j$ are two nonnegative integers with no carryover of $q$-adic digits in $i+j$, then
$\rho_\ast u^{i+j}=\rho_\ast u^i\cdot \rho_\ast u^j$ by Part 3 of Proposition \ref{sq2.1}. This is easily seen to be false in general when there
is carryover.

Now let $u^i/i!$ be the divided power element and $R\{\{u\}\}$ the $R$-algebra of formal divided power
series over $R$.

\begin{defn}\label{sq17}
We set
\begin{equation}\label{sq18}
\rho_\ast\left(\frac{u^i}{i!}\right):=\frac{u^{\rho_\ast i}}{(\rho_\ast i)!}\,,
\end{equation} and extend linearly to all of $R\{\{u\}\}$.\end{defn}

\begin{prop}\label{sq19}
The map $h\mapsto \rho_\ast h$ of Definition \ref{sq17} is an $R$-algebra automorphism
of $R\{\{u\}\}$. \end{prop}
\begin{proof}
Let $i,j$ be two nonnegative integers. Suppose first that  $i+j$ has no carryover of $q$-adic
digits. Then $\binom{i+j}{i}\equiv \binom{\rho_\ast (i+j)}{\rho_\ast i}\pmod p$ by Lucas. As
such, by the argument given after Definition \ref{sq15}, we see 
\begin{equation}\label{sq20}
\rho_\ast \left( \frac{u^i}{i!}\cdot \frac{u^j}{j!}\right)=\rho_\ast \left(\frac{u^i}{i!}\right)\cdot\rho_\ast
\left(\frac{u^j}{j!}\right)\,.\end{equation}
If there is carryover of digits, then $$\binom{i+j}{i}\equiv \binom{\rho_\ast i+\rho_\ast
j}{\rho_\ast i}\equiv 0\pmod{p}\,.$$
As such we again deduce the equality in Equation \ref{sq20} as both sides now vanish. 
\end{proof}

If $\mu$ is a measure on $\mathcal O$, we denote the measure $T_W^{-1}\rho T_W(\mu)$
by $\mu^{\rho_1}$;  we shall describe another action of $S_{(q)}$ on measures in Subsection
\ref{sq24}.

\subsubsection{The action of $S_{(q)}$ on $k_\infty$}\label{sq21}
Let $a\in A_+$ and let $s=(x,y)\in S_\infty$. Recall that in Definition \ref{series6} we defined
$a^{-s}$ through the use of the choice of positive uniformizer $\pi$.  Also, as mentioned in Part 2
of Proposition \ref{sq2.1}, the elements of $S_{(q)}$ stabilize both the nonnegative and nonpositive integers. This allows us to make the next definition. 

\begin{defn}\label{sq22}
Let $x\in k_\infty$ be written $\pi$-adically as $x=\sum_{j\gg -\infty} c_j\pi^j$ with $c_j\in \Fq$. 
Let $\rho_\ast\in S_{(q)}$. Then we set
\begin{equation}\label{sq23}
\rho_\ast x:=\sum_{j\gg-\infty} c_j\pi^{\rho_\ast j}\,.\end{equation}
\end{defn}

\noindent It is readily seen that $x\mapsto \rho_\ast x$ is a continuous automorphism of $k_\infty$
as $\Fq$-vector space. 

\subsubsection{Further actions on measures and functions}\label{sq24}
We will present here a number of other actions on functions and measures. As of this writing
we do not know how they are related.

Let $\mathcal O=\Fq[[\pi]]$ and $U\subset \mathcal O$ be the group of units and $U_1\subseteq
U$ the group of units $\equiv 1\pmod{(\pi)}$.
Let $\rho_\ast \in S_{(q)}$. Note that by definition $\rho_\ast 0=0$. The next result is
then obvious.
\begin{prop}\label{sq25}
As sets, $\mathcal O$,  $U$ and $U_1$ are stable under $\rho_\ast$.
\end{prop}

Let $X$ be one of $\{\mathcal O, U, U_1\}$. Let $f\colon X\to k_\infty$ be a continuous 
function. We then set $f^{\rho_1}(x):=f(\rho_\ast^{-1} x)$ (in analogy with our definition for functions on
$\Zp$); as before this is an action on the space of
continuous functions and is an automorphism of the algebra of such continuous functions.

One can use the Carlitz polynomials in $\Fq[\pi]$ (as in Subsection \ref{carpoly}) to obtain
a Banach basis for the continuous functions on $\mathcal O$. However, unlike the simple 
Corollary \ref{sq6}, we do not know how to easily compute $G_i^{\rho_1}(x)$ as of this writing.

As we have an action on the spaces of continuous functions, we deduce associated
actions on the space of measures. Again, let $X$ be as above and let $\mu$ be an $\mathcal
O$-valued measure on $X$.

\begin{defn}\label{sq26}
We let $\mu^{\rho_2}$ be the measure defined by 
\begin{equation}\label{sq27}
\int_X f(x)\, d\mu^{\rho_2} (x):=\int_X f(\rho_\ast^{-1} x)\, d\mu(x)\,.\end{equation}
\end{defn}

Let $f\colon \Zp \to \mathcal O$ be a Dirichlet series on $\Zp$. In Example \ref{sq13}
we established that, in general, $f^{\rho_1}(y)=f(\rho_\ast^{-1} y)$ will not be a Dirichlet
series on $\Zp$. 
On the other hand, by Theorem \ref{dirichlet4} we are guaranteed that $f(y)=\int_U u^y\, d\mu(u)$ for some
measure $\mu$.

\begin{defn}\label{sq28}
We set
\begin{equation}\label{sq29}
f^{\rho_2}(y):=\int_U u^y\, d\mu^{\rho_2}(u)\,.
\end{equation}\end{defn}
\noindent
Clearly $f\mapsto f^{\rho_2}(y)$ is an $\mathcal O$-linear automorphism of the space $\mathcal
D$ of Dirichlet series on $\Zp$.

Now let $X=\mathcal O$ and let $\mu \in \mathcal M(\mathcal O,\mathcal O)$. 
\begin{prop}\label{sq30}
The map $\mu\mapsto \mu^{\rho_2}$ is a automorphism of the convolution algebra of measures.
\end{prop}
\begin{proof}
This follows because the map $x\mapsto \rho_\ast x$ on $\mathcal O$ is $\Fq$-linear. \end{proof}

We have no idea of any relationship between the above action and the action
$\mu \mapsto \mu^{\rho_1}$
arising from the Wagner transform and  Proposition \ref{sq19}.

There is an obvious third action of $S_{(q)}$ on $\mathcal C(\Zp,\mathcal O)$ (including the
above action on Dirichlet series) 
by $f^{\rho_3}(y):=\rho_\ast f(y)$. If $f(y)=\sum_{j=0}^\infty c_j\binom{y}{j}$ then
\begin{equation}\label{sq31}
f^{\rho_2}(y)=\sum_{j=0}^\infty \rho_\ast (c_j) \binom{y}{j}\,.\end{equation}

Let $f\in \mathcal C(\mathcal O,\mathcal O)$ be continuous. We obtain a second action
of $S_{(q)}$ by $f^{\rho_2}(x):=\rho_\ast f(x)$. Further
let $f\colon \mathcal O \to \mathcal O$ have the Wagner expansion $\sum c_j G_j(x)$ where
$c_j\to o$ as $j\mapsto \infty$. We obtain $f^{\rho_2}(x)=\sum \rho_\ast c_jG_j(x)$. However,
we obtain yet a third $\Fq$-linear action on $\mathcal C(\mathcal O, \mathcal O)$ by 
\begin{equation}\label{sq32} 
f^{\rho_3}(x):=\sum \rho_\ast (c_j) G_j(x)\, ,\end{equation}
which is guaranteed to converge as $\rho_\ast c_j\to 0$ as $j\to \infty$ (and compare
Equations \ref{sq32} and \ref{sq31}). 

We leave to the reader the appropriate definitions for actions on the subspaces of $\Fq$-linear
functions.

Finally let $\mu\in \mathcal C(\mathcal O,\mathcal O)$ have expansion $\sum d_j \frac{u^j}{j!}$ dual
to the above Wagner expansions of continuous functions. We set
\begin{equation}\label{sq33}
\mu^{\rho_3}:=\sum \rho_\ast (d_j )\frac{u^j}{j!}\,.
\end{equation}
The map $\mu \to \mu^{\rho_3}$ is again $\Fq$-linear.

As of this writing, we do not know how these various actions interact.

\subsection{The action on $S_{(q)}$}\label{sq34}
In \cite{Go11} we defined an action of $S_{(q)}$ on $\mathbb S_\infty$ which we recall here. 
\begin{defn}\label{sq35}
Let $y\in \Zp$ and $\rho_\ast\in S_{(q)}$. Then we set 
\begin{equation}\label{sq36}
\hat\rho_\ast (y):=-\rho_\ast (-y)\,.\end{equation}
\end{defn}
\noindent
Clearly, Part 2 of Proposition \ref{sq2.1} implies that $\hat\rho_\ast$ also stabilizes both the nonnegative and nonpositive integers. 

Recall that by definition $\mathbb S_\infty=\C_\infty^\ast \times \Zp$. Let $L\subseteq \C_\infty$
be a subfield.

\begin{defn}\label{sq37}
We define 
\begin{equation}\label{sq38}
\mathbb S_{L,\infty}=L^\ast \times \Zp\subseteq \mathbb S_\infty\,.\end{equation}\end{defn}

\begin{defn}\label{sq39}
Let $s=(x,y)\in \mathbb S_{k_\infty,\infty}$. We set
\begin{equation}\label{sq40}
\rho_\ast s=(\rho_\ast x, \hat\rho_\ast y)\in \mathbb S_{k_\infty, \infty}\,.\end{equation}\end{defn}

Recall (see Remarks \ref{series9}) that, for an integer $i$,  we set $s_i=(\pi^{-i},i)\in \mathbb S_{k_\infty,\infty}$ with the property that $a^{s_i}=a^i$. From Definition \ref{sq38} we obtain
\begin{equation}\label{sq41}
\rho_\ast s_{-i}=s_{-\rho_\ast i}\,.
\end{equation}
In particular, if $i\equiv 0\pmod{q-1}$ then so is $\rho_\ast i$ (by Part 5 of Proposition \ref{sq2.1}).

In \cite{Sh98}, J.\ Sheats established that the zeroes of $\zeta_A(s)=\sum_{a\in A_+} a^{-s}$ all lie
in $k_\infty$ and are simple. The remarks just given imply, in particular, that our action on $\mathbb S_{k_\infty,\infty}$ 
permutes the {\it trivial zeroes} of $\zeta_A(s)$. For a further discussion of the interaction of
$S_{(q)}$ and the zeroes of $\zeta_A(s)$ we refer the reader to \cite{Go11}.

In general (for arbitrary $A$) there will be zeroes of $L$-functions which do not lie in
$k_\infty$ and, as such, we need to extend the action of $S_{(q)}$ further. On the other hand, one knows that {\em all} zeroes will be algebraic over $k_\infty$ and one {\em suspects} that, for a fixed $y\in \Zp$, the zeroes will lie in a finite extension of 
$k_\infty$. 

Therefore let $L\subset \C_\infty$ be a finite extension of $k_\infty$ and set $m=[L\colon k_\infty$]. Let $B=\{\alpha_1,\alpha_2,\ldots, \alpha_m\}$ be a $k_\infty$-basis of $L$ where we always assume
 $\alpha_1=1$. Via $B$, note that $L$ is isomorphic to $k_\infty^m$. 
 
 \begin{defn}\label{sq42}
 Let $\beta\in L$ be written $\sum_{e=0}^m k_e\alpha_e$ with $\{k_e\}\subset k_\infty$. We set
 \begin{equation}\label{sq43}
 \rho_\ast \beta:=\sum_e \rho_\ast (k_e )\alpha_e\,.\end{equation}
 We call this {\it the extended action of $S_{(q)}$  on $L$ given by $B$}.\end{defn}
 
\noindent
The reader will note that Definition \ref{sq42} is essentially the coordinate-wise action of
$\rho_\ast$ on $k_\infty^m$. 

Comparing the extensions given by two different bases of $L$ appears quite daunting. However,
there is one case where things simplify drastically. 

\begin{example}\label{sq45}
Let $\F_{q^m}\subset \C_\infty$ be the finite field of $q^m$ elements. Let $L:=k_\infty(\F_{q^m})$
so that $L\simeq \F_{q^m}\otimes_{\Fq}k_\infty$. Choose a basis  $B=\{1,\dots\}$ for $L/k_\infty$ 
consisting of elements in $\F_{q^m}$. Then the extension to $L$  of $\rho_\ast$ given in Definition \ref{sq42} is just the obvious $\F_{q^m}$-linear extension of $\rho_\ast$. As such
it is independent of the basis $B\subset \F_{q^m}$. 
\end{example}
Note that the class of constant field extensions is precisely the
class of unramified extensions of $k_\infty$.

\section{The Riemann Hypothesis as a ramification statement.}\label{rh}
Let $\zeta(s)=\sum_{n=1}^\infty n^{-s}$ be, as usual,  Riemann's zeta function with its meromorphic
continuation to the complex plane $\C$. Let $\Gamma(s)$ be Euler's gamma function and, again as usual, put
\begin{equation}\label{rh1}
\Xi(s):=1/2 \pi^{-s/2}s(s-1)\Gamma(s/2)\zeta(s)\,.
\end{equation}
Riemann showed that $\Xi(s)$ is entire and satisfies the functional equation $\Xi(s)=\Xi(1-s)$. The {\it Riemann Hypothesis} 
states that all zeroes of $\Xi(s)$ are of the form $s=1/2+it$ for some real number $t$. Following
Riemann, put $\hat\Xi(t):=\Xi(1/2+it)$, so that the functional equation becomes $\hat\Xi(t)=
\hat\Xi(-t)$ and the Riemann Hypothesis becomes the statement that the zeroes of
$\hat\Xi(t)$ are real. This ``reality'' statement is precisely echoed by the Theorem of Sheats
\cite{Sh98} mentioned above and underlies this section.

\subsection{The Riemann Hypothesis from the viewpoint of the Carlitz module}\label{rh2}

Let $\exp (z)=\sum\limits_{i=0}^\infty \frac{x^i}{i!}$ be the exponential function with period $2\pi i$. 
Note simply that
\begin{equation}\label{rh3}
\C=\R(2\pi i)\,.
\end{equation}
Now let $\exp_C(x)$ be the exponential of the Carlitz module with period $\bar \xi\in \C_\infty$.
Let $K=k_\infty(\bar\xi)$. As is well-known $K/k_\infty$ is Galois with Galois group isomorphic
to $\Fq^\ast=A^\ast$ (much as the Galois group of $\C/\R$ is isomorphic to $\Z^\ast=\{\pm1\})$. 
More importantly for our purposes, $K$ is {\it totally ramified} over $k_\infty$.

This analogy between $\C$ and $K$ suggests very strongly that we view (morally at least!)
$\C$ as being totally ramified over $\R$. Thus, from the optic of the Carlitz module,
the Riemann Hypothesis becomes
the statement ``the zeros of $\hat\Xi(t)$ are unramified."

\subsection{The Riemann Hypothesis in Characteristic $p$}\label{rh4}
Here we will be somewhat speculative as the subject is still far from mature. 
First of all, as noted, unramified extensions of $k_\infty$ are precisely the
constant field extensions.
Moreover, on the one hand,
we know from Sheats that the zeroes of $\zeta_{\Fq[\theta]}(s)$ are simple and in $k_\infty$;
on the other, we know from B\"ockle  \cite{Bo13} that there are other base rings $A$ where the zeroes
are not all in $k_\infty$ (but for a given $y\in \Zp$ in fact almost all of those computed are). 

Therefore, in general, we should {\em not} expect that the zeroes lie in $k_\infty$.

Suppose now that $\mathfrak k$ is a finite abelian extension of $k$ and for simplicity suppose
that $[\mathfrak k\colon k]$ is prime to $p$. Let $\mathfrak O\subset \mathfrak k$ be the ring
of $A$-integers. Let $I$ be a nontrivial ideal of $\mathfrak O$ and let $nI\in A_+$ be the monic
generator of the ideal theoretic norm of $I$. In the classical fashion we define
\begin{equation}\label{rh5}
\zeta_{\mathfrak O}(s):=\sum_{I\subset \mathfrak O} nI^{-s}\,.
\end{equation}
Furthermore, exactly as in classical theory, one has the factorization 
\begin{equation}\label{rh6}
\zeta_{\mathfrak O}(s)=\prod_{\psi}L(\psi,s)\,,
\end{equation}
where $\psi$ runs over all $\C_\infty$-valued characters of $\Gal(\mathfrak k/k)$ and
$L(\psi,s)$ is defined in the obvious fashion. Note that
(Chapter 8, \cite{go96}) all functions in Equation \ref{rh6} can be shown to be entire on $\mathbb S_\infty$.
Let $k_\infty(\psi)$ be the finite unramified (constant field) extension obtained by adjoining the
values of $\psi$. 

If the zeroes of $L(\psi,s)$ were all in $k_\infty$ then, by completeness, all the coefficients of
$L(\psi,s)$ would be also and this is {\em not} true in general. That is, the best one might hope for
is that all the roots lie in $k_\infty(\psi)$; i.e. are unramified.

\begin{rem}\label{rh7}
Suppose that we have an unramified zero.
Using Example \ref{sq45}, we see that the extended action of $S_{(q)}$ on the
zero is independent of the choice of base $B \subset \F_{q^m}$ chosen. Indeed, this independence may perhaps(!!) ultimately be one way to
establish the roots are in fact unramified.\end{rem}

So to summarize, as a first approximation, it seems that one might expect the roots of general
arithmetic $L$-series to be unramified. We know from examples that for general $A$ this will not
be true of all zeroes, but perhaps it does always hold for almost all zeroes for a given $y\in \Zp$. Of course
one wants to ultimately understand the arithmetic properties of {\em all} zeroes including the
ramified ones.


\begin{thebibliography}{3333334}
\bibitem[ADR16]{ADR16} {\sc Bruno Angl\`es, Tuan Ngo Dac, Floric Tavares Ribeiro}: 
Twisted characteristic $p$ zeta functions, (2016) arXiv:1603.04076.
\bibitem[An16]{An16} {\sc B.\ Angl\`es:} Private communication. (2016)
\bibitem[AP14]{AP14} {\sc B.\ Angl\`es, F.\ Pellarin}: Functional identities for $L$-series values in
positive characteristic. {\it J.\ Number Theory} 142 (2014) 223-251.
\bibitem[B\"o13]{Bo13} {\sc G.\ B\"ockle}: The distribution of the zeros of the Goss zeta function
for $A=\Fq[x,y]/(y^2+y+x^3+x+1)$. {\it Math.\ Z.} (2013) 835-861.

\bibitem[Con15]{Con15} {\sc A.\ Connes:} An essay on the Riemann Hypothesis, Preprint (2015)
arXiv:1509.05576.

\bibitem[Co00]{Co00} {\sc K.\ Conrad:} The digit prinicple. {\it J.\ Number Theory} 84 (2000) 230-257.

\bibitem[Dr74]{Dr74} {V.G.\ Drinfeld:} Elliptic modules. (Russian) {\it Mat.\ Sb.\ (N.S.)} 94(136) (1974) 594-627.

\bibitem[Go96]{go96}{\sc D.\ Goss:} Basic structures of function field arithmetic. {\it Ergebnisse der Mathematik und ihrer Grenzgebiete} (3) [Results in Mathematics and Related Areas (3)], 35. Springer-Verlag, Berlin, 1996. 
\bibitem[Go05] {Go05} {\sc D.\ Goss:} Applications of non-Archimedean integration to the $L$-series
of $\tau$-sheaves. {\it J.\ Number Theory} 110 (2005) 83-113.
\bibitem[Go11] {Go11} {\sc D.\ Goss:} Zeta phenomenology. {\it Noncommutative geometry, 
arithmetic, and related topics}, Johns Hopkins Univ. Press (2011) 159-182.
\bibitem[Go12] {Go12} {\sc D.\ Goss:} A local field approach to the Riemann hypothesis.
Preprint (2012) arXiv:1206.2040.

\bibitem[Ma58]{Ma58} {\sc K.\ Mahler:} An interpolation series for continuous functions of a
$p$-adic variable. {\it J.\ Reine Angew.\ Math.} 199 (1958) 23-34.
\bibitem[Po13]{po13} {\sc M.\ Papanikolas:} private correspondence, (2013).

\bibitem[Pe12]{Pe12} {\sc F.\ Pellarin:} Values of certain $L$-series in positive characteristic,
{\it Ann.\ of Math.} 176 (2012) 2055-2093.


\bibitem[Se62]{Se62} {\sc J.-P.\ Serre} Endomorphismes compl`etement continus des espaces
Banach $p$-adiques. {\it Publ.\ Math.\ IHES } 12 (1962) 69-85.
\bibitem[Sh98]{Sh98} {\sc J.\ Sheats:} The Riemann hypothesis for the Goss zeta function for
$\mathbf F_q[T]$. {\it J.\ Number Theory} 71 (1998) 121-157.
\bibitem[Si08]{Si08} {\sc W.\ Sinnott:} Dirichlet series in function fields. {\it J.\ Number Theory}
128 (2008) 1893-1899. 
\bibitem[Th95]{Th95} {\sc D.\ Thakur:} On characteristic $p$ zeta functions. {\it Comp.\ Math.}
99 (1995) 231-247.
\bibitem[Vo98]{vo98} {\sc J.\ Voloch:}  Differential operators and interpolation series in power
series fields. {\it J.\ Number Theory} 71 (1998) 106-108.
\bibitem[Wa70]{Wa70} {\sc C.\ Wagner:} Interpolation series for continuous functions on $\pi$-adic
completions of ${\rm GF}(q,x)$. {\it Acta Arith.} 17 (1970/1971) 389-406.


\end{thebibliography}
\end{document}